\documentclass[a4paper,11pt]{article}

\usepackage{amsmath}
\usepackage{amssymb}
\usepackage{amsthm}
\usepackage{graphicx}
\usepackage{psfrag}

\newtheorem{thm}{Theorem}[section]
\newtheorem{lemma}[thm]{Lemma}
\newtheorem{prop}[thm]{Proposition}

\newtheorem{defn}[thm]{Definition}

\newtheorem{rmk}[thm]{Remark}

\newcommand{\set}[1]{\left\{#1\right\}}

\newcommand{\abs}[1]{\left|#1\right|}

\newcommand{\mR}{\mathbb{R}}
\newcommand{\XX}{\mathcal{X}}

\newcommand{\dist}{\operatorname{dist}}

\title{Quantitative recurrence statistics and convergence to an extreme 
value distribution for non-uniformly hyperbolic dynamical systems}
\author{M.P. Holland, P. Rabassa, A.E. Sterk}
\date{September 2015}

\begin{document}

\maketitle

\begin{abstract}
For non-uniformly hyperbolic dynamical systems we consider the time 
series of maxima along typical orbits. Using ideas based upon 
quantitative recurrence time statistics we prove convergence
of the maxima (under suitable normalization) to an extreme value distribution, 
and obtain estimates on the rate of convergence. We show that our results are applicable to a range
of examples, and include new results for Lorenz maps, certain partially hyperbolic 
systems, and non-uniformly expanding systems with sub-exponential decay of correlations. 
For applications where analytic results are not readily available
we show how to estimate the rate of convergence to an extreme 
value distribution based upon numerical information of the quantitative recurrence statistics.
We envisage that such information will lead 
to more efficient statistical parameter estimation schemes based upon the block-maxima method.

\end{abstract}

\setcounter{tocdepth}{2}
\tableofcontents

\section{Introduction and background}\label{sec.intro-and-background}


\subsection{Extremes in dynamical systems}

Consider a dynamical system $(f,\XX, \nu)$, where $\XX$ is a $d$-dimensional
Riemannian manifold, $f:\XX\rightarrow \XX$ a measurable transformation, and
$\nu$ is an $f$-invariant probability measure. Assume that there is a compact
invariant set $\mathcal{A}\subseteq\XX$ which supports the measure $\nu$.
We let $B(x,r)=\{y:\dist(x,y)\leq r\}$ denote a closed ball in $\XX$ with respect
to the Riemannian metric $\dist(\cdot,\cdot)$. Given an observable
$\phi:\XX \rightarrow \mR$ we consider the stationary stochastic process
$X_1, X_2, \dots$ defined as
\begin{equation}\label{eq:phi_o_f} 
X_i =\phi \circ f^{i-1}, \quad i \geq 1,  
\end{equation}
and its associated maximum process $M_n$ defined as
\begin{equation}
\label{eq:max-process} 
M_n = \max(X_1,\dots,X_n). 
\end{equation}
Almost surely, $M_n\to\max_{x\in\mathcal{A}}\phi(x)$, and hence we are interested in the existence of
sequences $a_n,b_n\in\mathbb{R}$ such that
\begin{equation}\label{eq:extreme-conv1}
\nu\left\{x\in\XX: a_n(M_n-b_n) \leq u \right\}\to G(u),
\end{equation}
for some non-degenerate $G(u)$. The sequences $u_n:=u/a_n+b_n$ are chosen so that
\begin{equation}\label{eq:extreme-conv2}
\lim_{n\to\infty} n\nu\{x\in\XX: \phi(x)>u_n\}\to\tau(u),
\end{equation}
for some non-degenerate function $\tau(u)$. For the stochastic process defined
in \eqref{eq:phi_o_f}, our aim is to recover the same functions $G(u)$ as
computed in the case of independent identically distributed (i.i.d.)\ random
variables. In cases where this convergence holds, we also establish asymptotic bounds on 
the rate of convergence to $G(u)$ as $n\to\infty$. The case of i.i.d.\ random variables
has been widely studied, see \cite{Galambos, Leadbetter, Resnick}, and if the
limit function $G(u)$ is a non-degenerate distribution function then the limit
can only be of three following types:
\begin{description}
\item[Type I (Gumbel):]
\[
G(u) = \exp\left(- \exp \left[-\frac{u-b}{a} \right]\right), \quad -\infty < u < \infty;
\]

\item[Type II (Fr\'echet):]
\[
G(u) = \left\{ \begin{array}{ll} 
0, & u \leq b, \\ 
\displaystyle\exp\left(- \left[\frac{u-b}{a} \right]^{-\alpha}\right), &  u > b; 
\end{array} \right.
\]

\item[Type III (Weibull):]
\[
G(u) = \left\{ \begin{array}{ll} 
\displaystyle\exp\left(- \left[-\frac{u-b}{a} \right]^{\alpha}\right), &  u < b,\\ 
1, & u \geq b; 
\end{array} \right.
\]
\end{description}
for some parameters $a>0$, $b$ and $\alpha>0$. The
functional form of $G(u)$ in fact depends on $\tau(u)$, see \cite{Leadbetter}.
For example, in the case of i.i.d.\ random variables defined by the unit
exponential probability distribution $P$, we have  that $\tau(u)=e^{-u}$, and
$P(M_n\leq u+\log n)\to \exp(-e^{-u})$. Type II/III distributions arise in
the case where $\tau(u)$ has power law behaviour. Given a cumulative probability
distribution $G$, we say that $G$ follows an \emph{Extreme Value Distribution}
(EVD) if $G$ is any of the three distributions above. 

For weakly dependent stochastic processes satisfying equation \eqref{eq:extreme-conv2}, it was shown in 
\cite{Leadbetter} that convergence to a EVD is still valid (with the same 
distribution type as in the i.i.d.\ case) provided two probabilistic 
conditions $D(u_n)$ and  $D'(u_n)$ are shown to hold. 
In the dynamical systems setting,
much work (as we discuss below) has been devoted to finding conditions analogous to $D(u_n)$ and $D'(u_n)$ 
that ensure that \eqref{eq:max-process} converges to a EVD, and then checking that these 
conditions hold for a given system.


\subsection{Quantitative recurrence time statistics}

The aim of this article is to provide a general approach and formulate a collection of (checkable) 
conditions that a dynamical system must satisfy in order to ensure convergence to an EVD,
together with an estimate on the rate of convergence. The conditions we develop will be phrased in 
terms of i) the regularity of the observable; ii) the regularity of the invariant measure; 
iii) the rate of mixing of the dynamical system, and iv) 
quantitative asymptotics on the Poincar\'e recurrence time statistics.

We consider the class of observables that can be written in functional form
$\phi(x)=\psi(\dist(x,\tilde{x}))$ for some measurable function
$\psi:[0,\infty)\to\mathbb{R}$ taking its maximum at $0$ (hence $\phi$ is
maximized at $\tilde{x}$). When we speak of convergence to EVD, we will be
interested in the convergence for $\nu$-typical points $\tilde{x}\in\XX$ for which 
$\phi$ achieves its maximum. The conditions we develop will be applicable to a wide range of 
dynamical systems, including non-uniformly hyperbolic systems modelled by Young towers 
\cite{Young1,Young2}. This article will take forward and develop further the approaches 
used in (for example) \cite{Collet,Gupta,GHN,HNT,holland-nicol,Licheng}. In particular, a new
development that we discuss is on the rate of convergence to EVD for systems where there is a 
weak control on their quantitative recurrence statistics. 
As an application we establish a convergence rate to EVD for one and two-dimensional Lorenz maps, 
and formulate a conjecture on the convergence rate to EVD for the Lorenz flow, see \cite{GW,Lorenz63}. 
A second development is to optimize the approach used 
in \cite{Collet} and prove convergence to EVD (with a convergence rate) for non-uniformly 
expanding systems modelled by Young towers with stretched exponential decay of correlations. We 
also study rates of convergence to EVD for certain partially 
hyperbolic systems such those as considered in \cite{Gupta}. In particular we prove convergence to EVD for the
Alves-Viana map \cite{AV} together with an estimate on the convergence rate. 
As a further development, we also use quantitative recurrence statistics to achieve
numerical bounds on the rate of convergence to EVD when analytic information is not
readily available. For most of the hyperbolic systems that we study, we find that the (numerical) rate 
of convergence to EVD is power law provided our system observables are sufficiently 
regular. However, we will highlight cases where slow convergence to EVD is possible. We also
compare our methods to the approaches considered in \cite{CC,Haydn-Wasilewska,HNPV},
and study quantitative recurrence in situations where we don't expect a standard extreme 
law to hold, e.g. for quasi-periodic systems, \cite{Coelho, Coelho-F}.

To study extreme statistics via quantitative Poincar\'e recurrence statistics 
we consider a family of \emph{recurrence sets} $E_n$ defined as follows.  
Let $\tilde{g}:\mathbb{N}\to\mR$ be a
monotonically increasing function and for $n\ge 1$, let
\begin{equation}\label{eq:En}
    E_n:=\left\{ x\in\XX: \dist(x,f^jx)\le n^{-\frac{1}{d}},\,\textrm{for some}\;
j\in[1,\tilde{g}(n)]\right\},
    \end{equation}
where $d$ is the Euclidean dimension of the space. For some $\gamma<1$ we write $E_n:=E_n(\gamma)$
in the specific case $\tilde{g}(n)=n^{\gamma}$. 
The set $E_n$ captures the set of points
$x\in\XX$ which are approximately periodic (within distance $n^{-1/d}$) up to some time scale $\tilde{g}(n)$.
We also consider the set $E_n$ over other asymptotic time scales $\tilde{g}(n)$. Our specific interest is the study of the asymptotics of $\nu(E_n(\gamma))$ as $n\to\infty$. For hyperbolic dynamical systems we conjecture the following: there exists a
$\gamma_0\in(0,1)$ such that for all $\gamma<\gamma_0$,
\begin{equation}\label{eq:En-law}
\liminf_{n\to\infty}\frac{\log(\nu(E_n(\gamma))^{-1})}{\log
n}>0
\quad\mbox{and}\quad
\limsup_{n\to\infty}
\frac{\log(\nu(E_n(\gamma))^{-1})}{\log n}<\infty.
\end{equation}
i.e. $\nu(E_n(\gamma))\to 0$ as a power law. For certain non-uniformly expanding dynamical systems
we show that this condition holds. However, for general hyperbolic systems this estimate is difficult to prove analytically. Instead it is usually shown to hold for some shorter time scale $\tilde{g}(n)=(\log n)^{\gamma}$ with 
$\gamma>1$. Moreover, sometimes a weaker (sub-power law) asymptotic on $\nu(E_n)$ is achieved.  
From our main results, we will see that the main contribution to bounding the rate of convergence 
to an EVD comes from the asymptotic decay of $\nu(E_n(\gamma))$. The other main contribution 
to the bound comes from the correlation decay of the system over the time scale 
$\tilde{g}(n)=n^{\gamma}$, and hence we try to choose $\tilde{g}(n)$ tending to infinity as fast as possible.
In particular to achieve the best convergence rate we typically seek the largest possible 
$\gamma$ such that $\nu(E_n(\gamma))\to 0$ as a power law. We remark that we would not expect
$\nu(E_n)\to 0$ faster than a power law (unless the measure $\nu$ is highly irregular at periodic
points).

Let us be precise on how we characterize convergence to EVD. Define functions $\tau_n(u)$ and $G_n(u)$ by
\begin{equation}\label{eq.tauG}
\tau_n(u)=n\nu\{\phi(x)\geq u_n\},\quad G_{n^{a}}(u)=\left(1-\frac{\tau_n(u)}{n^{a}}\right)^{n^{a}},
\end{equation}
where $u_n:=u_n(u)$ is a sequence with the property that 
$\tau_n(u)$ converges to some $\tau(u)$ uniformly for all $u$ lying in a compact subset of $\mathbb{R}.$
For most of our applications this convergence property of $\tau_n(u)$ can be shown to hold for certain
linear sequences $u_n$ prescribed as in equation \eqref{eq:extreme-conv2}.
The parameter $a$ will be chosen as a fixed value in $(0,1)$. We consider the following two terms:
\begin{equation}\label{eq.error.b1b2}
\mathcal{B}_1(n):=\left|\nu\{M_n\leq u_n\}-G_{n^{a}}(u)\right|,\quad
\mathcal{B}_2(n):=\left|G_{n^{a}}(u)-G(u)\right|.
\end{equation}  
It follows that 
$$\left|\nu\{M_n\leq u_n\}-G(u)\right|\leq \mathcal{B}_1(n)+\mathcal{B}_2(n).$$
In this article we focus on finding a bound for the term $\mathcal{B}_1(n)$, and our main theorems will be based around
this quantity. We show that the main contribution to bounding $\mathcal{B}_1(n)$ comes from the speed at which $\nu(E_n)$ decays to zero. The rate of decay of correlations will also play a role in bounding $\mathcal{B}_1(n)$. For the term 
$\mathcal{B}_2(n)$, this will always go to zero if we assume that $\tau_n(u)$ converges to $\tau(u)$. 
From the limit definition of the exponential function the corresponding limit for $G(u)$ is then $e^{-\tau(u)}$. The rate at which it goes to zero certainly depends on the speed of convergence of $\tau_n(u)$ to $\tau(u)$. This latter convergence is not directly influenced by the recurrence statistics
nor on the rate of mixing. However, it does depend on the choice of the parameter $a\in(0,1)$, the sequence $u_n$, the regularity of the observable $\phi=\psi(\dist(x,\tilde{x}))$, and the regularity of the invariant measure (density) in the vicinity of the point $\tilde{x}\in\XX$. For i.i.d random variables we can take $a=1$, but 
for dependent processes the optimal value of $a$ tends to lie strictly inside $(0,1)$. We remark that even 
for i.i.d random variables the best bound for $\mathcal{B}_2(n)$ can be of the order $1/(\log n)$. 
This is true for Gaussian random variables, see \cite{Hall}. 

In the statement of our results, we focus on observables that are maximized at generic points $\tilde{x}\in\XX$. 
The exceptional set of points $\tilde{x}\in\XX$ where we cannot ensure convergence to an EVD (by our methods) has zero measure. However, the exceptional set is non-empty, and moreover contains all periodic points. For observables maximized at periodic points there has been study of their associated extreme statistics, see for example \cite{FFT3,FFT4}. 
We also focus on observables that can be expressed as smooth (or regularly varying) functions of the Euclidean metric $\dist(x,\tilde{x})$. For more general observables such as those considered in \cite{HVRSB}, an extended approach beyond this work is required in order to establish convergence to an EVD (with corresponding rates). 
The recurrence sets $E_n$ are naturally defined in terms of the level set geometries of $\phi(x)$ (i.e. balls). 
For more general observables, the definition of $E_n$ would need to be adapted to the geometry of the level set 
$\{\phi(x)=c\}$, for any $c\in\mathbb{R}$. Corresponding estimates on $\nu(E_n)$ would also need to be derived.

Relevant to this article we mention parallel approaches in studying extremes via return time statistics. 
For example, in \cite{FFT1} they show a direct correspondence between extreme value laws and return time 
distributions. This is extended in \cite{FFT2} to consider systems with non-smooth measures and/or 
observations. In \cite{FHN} they use quantitative recurrence statistics to show convergence to Poisson laws
for certain hyperbolic systems (e.g. billiards). For the systems we consider in this article we expect 
similar results to hold.

For results on numerical studies of convergence to an EVD see \cite{Alokley, Faranda1, Faranda2}. In these references 
they consider the performance of the numerical block-maxima approach as applied to time series data generated 
from certain chaotic dynamical systems. The block-maxima method is used as an intermediate step
to estimate the distribution parameters associated to the limiting EVD. The theoretical approach we use 
involves a blocking argument method, and in particular we gain an error estimate in terms of the block sizes 
and number blocks. With further work, we might expect to improve the convergence of the block maxima method 
by optimizing the choice block size and block length if specific knowledge of the time series is available 
(e.g. such as that of decay of correlations and/or the quantitative recurrence statistics).

This paper is organised as follows. In Section \ref{sec.statement.results} we state the main dynamical assumptions
and quantitative recurrence conditions. We then state our main results on convergence to EVD 
with corresponding convergence rates for a range of systems that include: non-uniformly expanding dynamical
systems (especially those under weak recurrence assumptions); non-uniformly expanding systems with 
sub-exponential decay of correlations; partially hyperbolic systems, and then non-uniformly hyperbolic systems.
In Section \ref{sec.general} we outline the main blocking argument approach and detail (via three key
propositions) how convergence to an EVD follows from the main dynamical assumptions such as the recurrence 
and mixing assumptions, and assumptions
on the regularity of the invariant measure. We then give the proofs of the main theorems. In Section 
\ref{sec.furtherdiscussion}
we give the link between our definition of $E_n$, and alternative definitions of quantitative recurrence that imply 
convergence to Poisson-like limit laws. We also discuss quantitative recurrence and corresponding convergence to EVD 
for continuous time flows, and discuss quantitative recurrence estimates for quasi-periodic systems.
In Section \ref{sec.numerics} we consider a range of case studies, including systems where analytic information on the
decay of $\nu(E_n)$ is not known. We show that typically $\nu(E_n)\to 0$ as a power law, and hence 
we might expect fast (power law) convergence to an EVD if we optimize the blocking method approach. 
Finally in Section \ref{sec.proofs} we prove the main technical results, such as the propositions stated
in Section \ref{sec.general}. 


\section{Statement of the main results}\label{sec.statement.results}
We consider first non-uniformly expanding dynamical systems, and the main assumptions here concern the existence 
of an ergodic invariant measure $\nu$ absolutely continuous with respect to Lebesgue measure, and
estimates on the quantitative recurrence and decay of correlations. We then
consider systems with sub-exponential decay of correlations and prove directly the existence of an extreme value law
(without making assumptions on the quantitative recurrence statistics). We then consider extreme statistics for
partially hyperbolic systems. Finally we consider extreme statistics for general non-uniformly hyperbolic systems. 

Throughout we fix the following notations. For positive functions $f(x)$ and $g(x)$,
we write $f(x)\sim g(x)$ if $f(x)/g(x)\to 1$ as $x\to\infty$. We say
$f(x)\approx g(x)$ if there exist $C_1,C_2$ such that $C_1\leq f(x)/g(x)\leq C_2$, and
$f(x)=O(g(x))$ if there is a constant $C>0$ such that $f(x)\leq Cg(x)$,
and $f(x)=o(g(x))$ if $f(x)/g(x)\to 0$. Similar statements apply for $x\to 0$, and in 
the case for functions replaced by sequences. 
\subsection{Convergence to an EVD for non-uniformly expanding systems}\label{sec.nue}
In this section we suppose that $(f,\XX,\nu)$ is a \emph{non-uniformly expanding system}. In particular 
let $\XX$ be a $d$-dimensional Riemannian manifold, and suppose that the
measure $\nu$ is absolutely continuous with respect to volume $m$. We assume that
$\nu$ admits a density $\rho\in L^{1+\delta}(m)$ for some $\delta>0$. In this
case the unstable dimension $\nu$-almost everywhere is equal to $d$, and
for all vectors $v\in T_x\XX$ we have
\[
\lim_{n\to\infty}\frac{1}{n}\log\|Df^n(x)v\|\geq\lambda_0>0
\quad
\textrm{for $\nu$-a.e. $x\in\XX$}.
\]
Under the assumption $\nu(E_n)\leq O(n^{-\alpha})$ convergence to an EVD was proved in \cite{HNT} with corresponding convergence rates established in \cite{holland-nicol}. Here we shall derive corresponding convergence results when bounds 
on the asymptotics of $\nu(E_n)$ are sub-polynomial. A natural application includes the study of 
one-dimensional Lorenz maps \cite{GHN}.  

We make the following dynamical assumptions. Recall that the function $\tilde{g}(n)$ and the
sets $E_n$ are defined in equation (\ref{eq:En}).  
\begin{enumerate}
\item[(H1)]{\bf (Decay of correlations).}
There exists a monotonically decreasing sequence $\Theta(j)\to 0$ such that
for all $\varphi_1$ Lipschitz continuous and all $\varphi_2 \in L^{\infty}$:
\[
\left|\int \varphi_1 \cdot \varphi_2\circ f^j d \nu -\int \varphi_1 d \nu \int \varphi_2 d\nu\right|
\leq
\Theta(j) \|\varphi_1\|_{\textrm{Lip}} \|\varphi_2\|_{L^\infty},
\]
where $\|\cdot\|_{\textrm{Lip}}$ denotes the Lipschitz norm. (For non-uniformly expanding maps,
this will be our decay assumption).
\item[(H2a)]{\bf (Strong quantitative recurrence rates).}
There exist numbers $\gamma, \alpha>0$ and $C>0$ such that:
\begin{equation}\label{eq:h2a-rec}
\tilde{g}(n)\sim n^{\gamma}
\quad\implies\quad
\nu(E_n)\leq\frac{C}{n^{\alpha}}.
\end{equation}
\item[(H2b)]{\bf (Weak quantitative recurrence rates).}
There exist numbers $\gamma>1$, $\alpha>0$ and $C>0$ such that:
\begin{equation}\label{eq:h2b-rec}
\tilde{g}(n)\sim (\log n)^{\gamma}
\quad\implies\quad
\nu(E_n)\leq\frac{C}{(\log n)^{\alpha}}.
\end{equation}
\end{enumerate}
We remark on these conditions as follows. To prove convergence to an EVD
for systems having polynomial decay of correlations
we require the condition in \eqref{eq:h2a-rec} to hold. For systems
having exponential decay of correlations we can prove convergence
to an EVD under milder assumptions on the recurrence conditions, and in particular
we can assume the weaker version \eqref{eq:h2b-rec}. In particular the choice of function $\tilde{g}(n)$ plays
an important role via a control of the asymptotics of $\Theta(\tilde{g}(n))$. As well as conditions
(H2a) or (H2b), the rate of convergence to an EVD is also linked to how fast $\Theta(\tilde{g}(n))\to 0$,
see Section \ref{sec.general}.
Hence it is possible to formulate weak recurrence in terms of other asymptotic forms
for $\tilde{g}(n)$ and $\nu(E_n)$, but from a point of view of applicability
we will not optimize beyond stating (H2b). 
Intermediate versions include having $\tilde{g}(n)\sim (\log n)^{\gamma}$ 
imply $\nu(E_n)\leq Cn^{-\alpha}$ (see \cite{Collet, holland-nicol}), and 
also the following which we state formally as (H2c): 
\begin{enumerate}
\item[(H2c)]{\bf (Intermediate quantitative recurrence rates).}
For some $\gamma>1$, $\alpha\in(0,1)$ and $C>0$:
\begin{equation}\label{eq:h2c-rec}
\tilde{g}(n)\sim (\log n)^{\gamma}
\quad\implies \nu(E_n)\leq C\exp\{-(\log n)^{\alpha}\}.
\end{equation}
\end{enumerate}
Condition (H2c) has been verified for Lorenz maps, see \cite{GHN,Licheng}.
In Section \ref{sec.proofstretched} we prove that condition (H2c) holds for non-uniformly expanding
systems that have stretched exponential decay of correlations. An
application of this includes extremes for certain partially hyperbolic
systems such as the Alves-Viana map \cite{AV}, as discussed in
in \cite{Gupta}.

The first result we state concerns convergence to an EVD for observables 
$\phi:\XX\to\mathbb{R}$ that take their maxima at some specified $\tilde{x}\in\XX$. In particular the result
below generalizes that of \cite{HNT,holland-nicol} in the case of systems 
satisfying condition (H2b) or (H2c). Unless stated otherwise we consider linear sequences $u_n(u)$, and 
relative to equation \eqref{eq.tauG} we take
\begin{equation}\label{eq.tauG.linear}
\tau_n(u)=n\nu\left\{\phi(x)\geq \frac{u}{a_n}+b_n\right\},\quad G_{\sqrt{n}}(u)=\left(1-\frac{\tau_n(u)}{\sqrt{n}}
\right)^{\sqrt{n}}.
\end{equation}
The choice $a=1/2$ in the form of $G_{n^a}(u)$ turns out to be optimal in most situations we consider. 
We shall consider observable functions of the form 
$\phi(x)=\psi(\dist(x,\tilde{x}))$, where $\psi:\mathbb{R}^+\to\mathbb{R}$ 
takes its maximum at zero. We will assume that $\psi$ is monotonically decreasing and is regularly varying at zero. To keep the exposition simple, we will take the explicit case $\psi(y)=-\log y$. We have the following result

\begin{thm}\label{thm.extreme}
Suppose that $f: \XX \to \XX$ is a non-uniformly expanding map with ergodic
measure $\nu$ having density $\rho\in L^{1+\delta}(m)$ for some $\delta>0$. 
Consider the observable function $\phi(x)=-\log(\dist(x,\tilde{x}))$, and suppose in equation
\eqref{eq.tauG.linear} we take sequences $a_n=d$, $b_n=d^{-1}\log n$. Then we have
the following:
\begin{enumerate}
\item 
Suppose there exists $\theta_0<1$ and $\beta\in(0,1]$ 
such that $\Theta(n)=O(\theta^{n^{\beta}}_0)$ 
and (H1) holds together with (H2b). Assume further that $\gamma\beta>1$.
Then for $\nu$-a.e. 
$\tilde{x}\in \XX$ we have
\begin{equation}\label{eq.thm1.h2b}
\left| \nu\{M_n\leq u_n\}- G_{\sqrt{n}}(u)\right|\leq 
\frac{C_1}{(\log n)^{\alpha-\kappa}}
,\quad\textrm{for any}\;
\kappa>\frac{C_2}{\beta},
\end{equation}
where $C_1>0$ depends on $\tilde{x}$ and $C_2>0$ depends $\delta$. 
\item 
Suppose there exists $\theta_0<1$ and $\beta\in(0,1]$ 
such that $\Theta(n)=O(\theta^{n^{\beta}}_0)$ 
and (H1) holds together with (H2c). Assume further that $\gamma\beta>1$. Then there exists $\tilde{\alpha}>0$
such that for $\nu$-a.e. $\tilde{x}\in \XX$ we have
\begin{equation}\label{eq.thm1.h2c}
\left| \nu\{M_n\leq u_n\}- G_{\sqrt{n}}(u)\right|\leq 
C_1\exp\{-(\log n)^{\tilde\alpha}\},
\end{equation}
where $C_1$ depends on $\tilde{x}$.
\end{enumerate}
Furthermore, for $\nu$-a.e. $x\in\XX$, there exists $C(\tilde{x})>0$
such that
\begin{equation}\label{eq.Mn.gumbel}
\lim_{n\to\infty}\nu\{M_n\leq u_n\}=\exp\{-C(\tilde{x})e^{-u}\}.
\end{equation}
\end{thm}
\begin{rmk}
In the case where (H1) holds and $\Theta(n)=O(n^{-\zeta})$ for some
$\zeta>0$ a corresponding estimate is derived in \cite{holland-nicol}
for the same observable function type, and for densities $\rho\in L^{1+\delta}.$ 
Under the assumption of (H2a) it is shown that for $\nu$-a.e. $\tilde{x}\in\XX$:
\begin{equation}
\left| \nu\{M_n\leq u_n\}- G_{\sqrt{n}}(u)\right|\leq\frac
{C_1}{n^{\frac{1}{2}-\kappa}}+ \frac{C_2}{
n^{\alpha-\kappa}}
,\quad\textrm{for any}\;
\kappa>\frac{C_3}{\zeta\delta},
\end{equation}
where $C_1,C_2$ and $C_3$ depend on $\tilde{x}$.
\end{rmk}
\begin{rmk}
For non-uniformly expanding systems, the constant $C(\tilde{x})$ in equation
\eqref{eq.Mn.gumbel} is determined by the value of the density $\rho$ at $\tilde{x}$. From the
functional form of $\phi(x)$, and the choice $a_n=d$ and $b_n=d^{-1}\log n$ we have
by Lebesgue differentiation:
\begin{equation}\label{eq.tau.conv}
\tau_n(u)=n\nu\left\{x\in\XX:\dist(x,\tilde{x})\leq\frac{e^{-u/d}}{n^{1/d}}\right\}\to\rho(\tilde{x})e^{-u},\;(n\to\infty). 
\end{equation}
In particular, control of the the error term $\mathcal{B}_2(n)$ in equation \eqref{eq.error.b1b2} depends 
on the convergence rate of $\tau_n(u)$ to $\tau(u)$ in equation \eqref{eq.tau.conv}. To bound
this rate additional regularity conditions on $\rho$ are required (such as H\"older continuity). 
\end{rmk}
We remark further that the error bound in equation \eqref{eq.thm1.h2b} is of little utility if
the constant $\alpha$ in (H2b) is small relative to $1/\beta$. The constant $\tilde\alpha$
in equation \eqref{eq.thm1.h2c} can in fact be chosen arbitrarily close to (but less than) $\alpha$.

Theorem \ref{thm.extreme} also extends to other observable types, such as those
of the form $\phi(x)=\psi(\dist(x,\tilde{x}))$, where $\psi(y)$ is a regularly
varying function taking its maximum at $y=0$. For example, suppose for some
positive function $\eta(u)$ we have
\begin{equation}\label{eq.sv}
\lim_{u\to 0}\frac{\psi^{-1}(u+\ell \eta(u))}{\psi^{-1}(u)}=e^{-\ell}.
\end{equation}
Then we get convergence to Type I in Theorem \ref{thm.extreme}. For corresponding conditions that lead
to convergence to Type II or III see \cite{FFT1,HNT} for further  details.

As an application of Theorem \ref{thm.extreme} we derive a bound on the rate of 
convergence to an EVD for the class of one-dimensional expanding \emph{Lorenz maps} considered in \cite{GHN}. 
These maps arise 
naturally out of the construction of the Poincar\'e map of the Lorenz equations \cite{Lorenz63}. 
See also Section \ref{sec.lorenz}. The one dimensional \emph{Lorenz map} $f:\XX\to \XX$ 
(with $\XX=[-1,1]$) satisfies the following conditions:
\begin{itemize}
\item[(L1)] There exist $C>0$ and $\lambda>1$ such that for all
$x\in \XX$ and $n>0$, $|(f^n)'(x)|>C\lambda^n.$
\item[(L2)] There exist $\beta',\beta\in(0,1)$ such that 
$f'(x)=|x|^{\beta-1}g(x)$
where $g\in C^{\beta'}(\XX)$, $g>0$.
\item[(L3)] $f$ is locally eventually onto, i.e. for all intervals
$J\subset \XX$, there exists $k=k(J)>0$ such that $f^k(J)=\XX$.
\end{itemize}
Notice that $f$ has a derivative singularity at $x=0$. We have the following
result.
\begin{thm}\label{thm.lorenz1}
Suppose that $f:\XX\to\XX$ is an expanding Lorenz map satisfying
(L1)-(L3), and suppose that $\phi(x)=-\log(\dist(x,\tilde{x}))$. Then there 
exists $\alpha>0$ such that for $\nu$-a.e. $x\in\XX$ we have:
\[
\left| \nu\{M_n\leq u+\log n\}- \exp\{-C(\tilde{x})e^{-u}\}\right|\leq 
O(1)\exp\{-(\log n)^{\alpha}\},
\]
where $C(\tilde{x})$ depends on the invariant density
at $\tilde{x}$.
\end{thm} 
We remark that this theorem provides us with an estimate on the combined error terms $\mathcal{B}_1(n)$
and $\mathcal{B}_2(n)$ as specified in equation \eqref{eq.error.b1b2}.

\subsection{Convergence to an EVD for systems with stretched exponential 
mixing rates}\label{sec.stretched}
In this section we establish convergence to an EVD for systems with stretched 
exponential mixing rates. We will assume that $(f,\XX,\nu)$ is a non-uniformly
expanding system modelled by a Young tower with a stretched-exponential return
time asymptotic. We summarize the tower model as follows, see \cite{Young1, Young2}.
We suppose that there is a set 
$\Lambda\subset\XX$ together with a countable partition into subsets
$\{\Lambda_i\}$. Let $R:\Lambda\to\mathbb{N}$ be an $L^1(m)$ roof function 
with the property that 
\[
R|_{\Lambda_{l}}: = R_l,\quad\forall\Lambda_l\subset\Lambda,
\]
and $f^{R_i}\Lambda_i=\Lambda$ (modulo sets of Lebesgue measure zero). 
The Young Tower is given by 
\[
\Delta = \underset{i, l \le R_i-1}{\bigcup}\{ (x, l) : x\in \Lambda_{i }\},
\] 
and the  tower map $F:\Delta \to \Delta $ by
\[
F(x, l) = \begin{cases}(x, l+1) & \mbox{ if }  x\in \Lambda_{i}, l<R_i-1\\
(f^{R_i}x, 0) &\mbox{ if } x\in \Lambda_{i}, l = R_i-1\end{cases}.
\]
Define the map $\widehat{F}=F^R:\Lambda\to\Lambda$. Under further hypotheses
on $\widehat{F}$, such as bounded distortion, it is shown that 
$\widehat{F}$ preserves an invariant ergodic measure $\nu_{0}$ which
is uniformly equivalent to $m$. The statistical properties of $f$ such
as mixing rates can be determined from the asymptotics of
$m\{x\in\Lambda:R(x)>n\}.$  We have the following result.
\begin{thm}\label{thm.stretched}
Suppose that $f: \XX \to \XX$ is a non-uniformly expanding map 
modelled by a Young tower over a base set $\Lambda$, 
and for constants $\theta_0,\beta<1$
assume that $m\{x\in\Lambda: R(x)>n\}=O(\theta^{n^{\beta}}_0)$. Assume further that $\|Df\|_{\infty}<\infty$.
Consider the observable function $\phi(x)=-\log(\dist(x,\tilde{x}))$, and suppose in equation
\eqref{eq.tauG.linear} we take sequences $a_n=d$, $b_n=d^{-1}\log n$.
Then for $\nu$-a.e.  
$\tilde{x}\in \XX$ we have
\begin{equation}\label{eq.thm-weak}
\left| \nu\{M_n\leq u_n\}- G_{\sqrt{n}}(u)\right|\leq 
O(1)\exp\{-(\log n)^{\tilde\beta}\},
\end{equation}
where $\tilde\beta>0$ is independent of $n$, but dependent 
on $\tilde{x}$. Furthermore for $\nu$-a.e. 
$\tilde{x}\in \XX$, there exists $C(\tilde{x})$ such that:
\[
\nu\{M_n\le u_n\}\to \exp\{-C(\tilde{x})e^{-u}\}.
\]
\end{thm}
We shall prove this theorem in Section \ref{sec.proofstretched}. The result optimizes the arguments
developed in \cite{Collet}. Along the way we show that such systems satisfy condition (H2c), and we
obtain an estimate on the regularity of the invariant density for systems having sub-exponential
decay of correlations. The abstract approach we adopt does not appear to extend to systems having polynomial
decay of correlations. However, certain systems with polynomial decay of correlations are known to admit
extreme value laws, see \cite{HNT} in the case of non-uniformly expanding systems, and see \cite{Haydn-Wasilewska}
in the case of hyperbolic systems. In the latter, they gain an error estimate of logarithmic type. 
We will discuss further hyperbolic systems in Section \ref{sec.hyperbolic}.

\subsection{Convergence to an EVD for partially hyperbolic systems}\label{sec.partial}

In this section we study convergence to an EVD for partially hyperbolic 
systems and establish error rates for convergence
under weak or strong assumptions on the quantitative recurrence statistics.
The results we state build upon the work of \cite{Gupta}, and we solve
a question posed within on the existence of an EVD limit law for the 
Alves-Viana map \cite{AV}.

Suppose that $Y$ is a compact, 
$d_Y$-dimensional manifold with metric $\dist_{Y}$ and $X$
is a compact $d_X$-dimensional manifold with metric $\dist_X$ . We let $d=d_X+d_Y$ and define a
metric on $X\times Y$ by
$$d((x_1,\theta_1),(x_2,\theta_2))=\sqrt{\dist_{X}(x_1,x_2)^2+\dist_Y(\theta_1,\theta_2)^2}.$$
We denote the Lebesgue measure on X by $m_{X}$, the Lebesgue measure on $Y$ by $m_Y$ and the
product measure on $X\times Y$ by $m=m_X\times m_Y$.

If $T:X\to X$ is a measurable transformation and $u:X\times Y\to Y$ a measurable
function, then we can define a $Y$-skew extension of $T$ by $u$, via the map:
$f:X\times Y\to X\times Y$:
\begin{equation}\label{eq.partialsystem}
f(x,\theta)=(T(x),u(x,\theta)).
\end{equation}
We assume further that $T:X\to X$ has an ergodic invariant measure $\nu_X$,
and $f$ preserves an invariant probability measure $\nu$ with density in
$\rho\in L^{1+\delta}(m)$. Given $(\tilde{x},\tilde\theta)$, we consider 
the observable function 
$\phi(x,\theta)$ with representation 
$$\phi(x,\theta)=\psi(\dist((x,\theta),(\tilde{x},\tilde\theta))),$$
where $\psi:\mathbb{R}^{+}\to\mathbb{R}$ takes it maximum value at $0$.
For partially hyperbolic systems we phrase assumptions (H2a)-(H2c) in
terms of the measure $\nu_{X}$ and the recurrence set $E^{X}_n$, where
$$E^{X}_{n}:=\left\{x\in X:\,\dist_X(T^jx,x)<n^{-\frac{1}{d_X}},\;
\textrm{some}\;j\in[1,\tilde{g}(n)]\right\}.$$

\begin{thm}\label{thm.partial} 
Suppose that $(f,X\times Y,\nu)$ is a partially hyperbolic system
with representation given in equation \eqref{eq.partialsystem}. Assume that 
$\rho\in L^{1+\delta}$ for some $\delta>0$.
Consider the observable function  $\phi(x)=-\log(\dist((x,\theta),(\tilde{x},\tilde{\theta})))$, 
and suppose in equation \eqref{eq.tauG.linear} we take sequences $a_n=d$, $b_n=d^{-1}\log n$. 
Then we have the following.
\begin{enumerate}
\item
Suppose that $\Theta(n)=O(\theta_0^{n^{\beta}})$ for some $\theta_0<1$
and $\beta\in(0,1]$, and condition (H1) holds. Suppose also
that condition (H2b) holds for the set $E^{X}_n$ and the measure $\nu_X$.
Assume further that $\gamma\beta>1$. Then for all $\epsilon>0$ and $\nu$-a.e. 
$(\tilde{x},\tilde{\theta})\in\XX$ we have that
\begin{equation}\label{eq.partial-weak}
\left| \nu\{M_n\leq u_n\}- G_{\sqrt{n}}(u)\right|\leq 
\frac{C_1}{(\log n)^{\alpha-\kappa}},\quad\textrm{for any}\;\kappa>\frac{C_2}{\beta},
\end{equation}
where $C_1>0$ depends on $(\tilde{x},\tilde{\theta})$, and $C_2>0$ depends on $\delta$.
\item
Suppose that  $\Theta(n)=O(n^{-\zeta})$ for some $\zeta>0$  and (H1) holds. 
Suppose also that condition (H2a) holds for the set $E^{X}_n$ and the 
measure $\nu_X$. Then for $\nu$-a.e. $(\tilde{x},\tilde{\theta})\in\XX$ 
we have that \begin{equation}\label{eq.partial-strong}
\left| \nu\{M_n\leq u_n\}- G_{\sqrt{n}}(u)\right|\leq 
\frac{C_1}{n^{1/2-\kappa}}+ \frac{C_2}{
n^{\alpha-\kappa}},\quad\textrm{for any}\;
\kappa>\frac{C_3}{\zeta\delta},
\end{equation}
where $C_1,C_2>0$ depend on $(\tilde{x},\tilde{\theta})$, and $C_3>0$.
\end{enumerate}
Moreover in both cases above we have for some $C(\tilde{x},\tilde{\theta})>0$:
\begin{equation}
\lim_{n\to\infty}\nu(M_n<u_n)=\exp\{-C(\tilde{x},
\tilde{\theta})e^{-u}\}.
\end{equation}
\end{thm}
As an application, we consider the Alves-Viana map 
$f:S^1\times I\to S^1\times I$ given by 
\begin{equation}\label{eq.av}
f(x,y)=(dx\mod 1,a_0+\eta(x)-x^2),
\end{equation}
where in \cite{AV} they explicitly take $d=16$ and 
$\eta(x)=\epsilon\sin(2\pi x)$ for $\epsilon$ sufficiently small. 
The parameter $a_0$ is chosen so that the point $x=0$ is pre-periodic under 
the map $x\mapsto a_0-x^2$. It is shown in \cite{gouezel} that $(f,\XX\,u)$ 
is modelled by a Young tower with $m\{R>n\}\leq C\theta_0^{n^{\beta}}$ for some 
$\beta,\theta_0<1$. We obtain:
\begin{thm}\label{thm.alvesviana}
Suppose $(f,\XX,\nu)$ is the Alves-Viana map given by equation \eqref{eq.av},
and we take observable function $\phi(x)=-\log \dist((x,\theta),(\tilde{x},\tilde\theta))$
in equation \eqref{eq.tauG.linear}. Then for all $\epsilon>0$ and $\nu$-a.e. $(\tilde{x},\tilde{\theta})
\in\XX$:
\begin{equation}\label{eq.thm-av}
\left| \nu\left\{M_n\leq\frac{1}{2}(u+\log n)\right\}- G_{\sqrt{n}}(u)\right|\leq 
C((\tilde{x},\tilde\theta))\frac{1}{n^{1/2-\epsilon}},
\end{equation}
where $C>0$. Moreover, for $\nu$-a.e. $\tilde{x}\in\XX$, and some $C_1(\tilde{x},\tilde{\theta})>0$ we have
\begin{equation}
\lim_{n\to\infty}\nu\left(M_n<\frac{1}{2}(u+\log n)\right)=\exp\{-C_1e^{-u}\}.
\end{equation}
\end{thm}
Notice that the bound on the error rate is sharper than that established
in Theorem \ref{thm.stretched}. In fact the error rate comes from
the rate associated to $\nu_X(E^{X}_n)$ using Theorem \ref{thm.partial}. 
However, as was pointed out in \cite{Gupta}, it is not known that the density
$\rho$ belongs to $L^p$ (for $p>1$). We get round this issue via a weaker
bound on its regularity, and show that the conclusion of 
Theorem \ref{thm.partial} is still applicable.


\subsection{Convergence to an EVD for non-uniformly hyperbolic systems}\label{sec.hyperbolic}
For non-uniformly hyperbolic systems, we suppose that $\nu$ is a
Sinai-Ruelle-Bowen (SRB) measure, and $(f, \XX ,\nu)$ is a non-uniformly hyperbolic system
modelled by a Young tower \cite{Young1}. Relative to non-uniformly expanding systems
we need a version of (H1) restricted to the class of Lipschitz functions. We state the following assumption: 
\begin{enumerate}
\item[(H1s)]{\bf (Decay of correlations).}
There exists a monotonically decreasing sequence $\Theta(j)\to 0$ such that
for all Lipschitz $\varphi_1$  and $\varphi_2$:
\[
\left|\int \varphi_1 \cdot \varphi_2\circ f^j d \nu -\int \varphi_1 d \nu \int \varphi_2 d\nu\right|
\leq
\Theta(j) \|\varphi_1\|_{\textrm{Lip}} \|\varphi_2\|_{\textrm{Lip}},
\]
where $\|\cdot\|_{\textrm{Lip}}$ denotes the Lipschitz norm.
\end{enumerate}
For non-uniformly hyperbolic systems the measure $\nu$ need not be
absolutely continuous with respect to Lebesgue measure. Its regularity can
be determined by using local dimension estimates. Recall that the 
pointwise local dimension of $\nu$ is given by:
\begin{equation}
d_{\nu}:=\lim_{r\to 0}\frac{\log\nu(B(x,r))}{\log r},
\end{equation}
whenever this limit exists. For the examples we consider the local dimension 
of $\nu$ 
exists for $\nu$-a.e. $x\in\XX$. However, we also need control on the regularity of $\nu$ on certain
shrinking annuli. We state the following assumption (H3):
\begin{enumerate}
\item[(H3)]{\bf (Regularity of $\nu$ on shrinking annuli).} 
For all $\delta>1$ and $\nu$-a.e.$x\in\XX$, there exists $\sigma>0$ such that
\begin{equation}
|\nu(B(x,r+r^{\delta}))-\nu(B(x,r))|\leq Cr^{\sigma\delta}.
\end{equation}
The constant $C$ and $\sigma$ depending on $x$ (but not $\delta$).
\end{enumerate}
To state our result, we take explicitly the observation $\psi(u)=-\log u$.
Analogous results hold for other functional forms, such as the case where
$\psi(u)$ is regularly varying at $u=0$, see \cite{HNT,FFT1}.
\begin{thm}\label{thm.extreme2}
Suppose that $(f,\XX,\nu)$ is a non-uniformly hyperbolic system modelled by a Young tower with
SRB measure $\nu$. Suppose the local dimension $d_{\nu}$ exists
for $\nu$-a.e. $x\in\XX$, and (H3) holds. 
Consider the observable $\phi(x)=-\log(\dist(x,\tilde{x}))$,
and suppose that $u_n:=u_{n}(u)$ is a sequence such that 
$\limsup_{n\to\infty}n\nu\{\phi(x)>u_n\}=\tau(u)<\infty.$
\begin{enumerate}
\item 
Suppose that $\Theta(n)=O(\theta^{n^{\beta}}_0)$ for some $\theta_0<1,\beta\in(0,1]$ 
and (H1s) holds together with (H2b), with $\gamma\beta>1$. 
Then for all $\epsilon>0$, and for $\nu$-a.e. $\tilde{x}\in \XX$ we have
\[
\left| \nu\{M_n\leq u_n\}- G_{\sqrt{n}}(u)\right|\leq 
\frac{C_1}{(\log n)^{\alpha-\kappa}},\quad\textrm{for any}\;
\kappa>\frac{C_2}{\beta},
\]
where $C_1>0$ depends on $\tilde{x}$ and $C_2>0$ depends on $\sigma$. 
\item Suppose that for $\zeta>0$, $\Theta(n)=O(n^{-\zeta})$ and (H1) 
holds together with (H2a). Then, for all $\epsilon>0$, and $\nu$-a.e. 
$\tilde{x}\in \XX$:
\[
\left| \nu\{M_n\leq u_n\}- G_{\sqrt{n}}(u)\right|\leq 
\frac{C_1}{n^{\frac{1}{2}-\kappa}}+
\frac{C_2}{
n^{\alpha-\kappa}}
,\quad\textrm{for any}\;
\kappa>\frac{C_3}{\zeta\sigma},
\]
where $C_1,C_2>0$ depend on $\tilde{x}$ and $C_3>0$.
\end{enumerate}
\end{thm}
We make several remarks on this theorem. The first remark concerns
the sequence $u_n$ and whether we have convergence to EVD.
From the definition of local dimension, we know that $\forall\epsilon>0$:
\begin{equation}\label{eq.srb-dim}
r^{d_{\nu}-\epsilon}\leq\nu(B(x,r))\leq r^{d_{\nu}+\epsilon}.
\end{equation}
This is the best that can be achieved, and is a weaker statement than 
achieving an asymptotic of the form $\nu(B(x,r))\sim \ell(r)r^{d_{\nu}}$ 
(for some slowly varying function $\ell(r)$). Hence for linear sequences 
of the form $u_n=u/a_n+b_n$, the function $G_{\sqrt{n}}(u)$ 
need not converge to one of the standard EVD types I-III, see \cite[Section 1.6]{Leadbetter}. 
To get convergence to EVD For the observable $\phi(x)=-\log(\dist(x,\tilde{x}))$ then $u_n$ will be some
(non-linear) sequence satisfying:
$$u_n\in\left[\frac{(1-\epsilon)}{d_{\nu}}(u+\log n),\frac{(1+\epsilon)}{d_{\nu}}(u+\log n)\right].$$  
We remark that this theorem is based upon the definition of the recurrence
set $E_n$ and the asymptotic properties of $\nu(E_n)$. These properties should 
be contrasted to the \emph{short return time} (SRT) conditions that 
form the basis of the results presented \cite[Section 5]{holland-nicol},
and in \cite{CC,HNPV,Haydn-Wasilewska}. The link between assumptions (H2a)-(H2c) 
and the SRT conditions will be discussed further in Section \ref{sec.SRT}. 

Theorem \ref{thm.extreme2} is proved in Section \ref{sec.thm.proofs},
and is applicable to hyperbolic billiards, Lozi maps, solenoid maps, and
certain non-uniformly hyperbolic dynamical systems such as the H\'enon map.
These examples have been discussed on a case-by-case basis in the aforementioned references.
However, Theorem \ref{thm.extreme2} builds upon these works in the case of weak quantitative recurrence 
statistics. To keep the exposition simple, we did not include a precise statement in the case of assumption (H2c), 
but an analogous statement applies. In fact, based upon a recent result of \cite{Licheng}, we can deduce an estimate 
on the convergence rate to EVD for the family of two-dimensional Poincar\'e return maps associated to 
the geometric Lorenz flow, \cite{GW}. 
For such a system $(f,\XX,\nu)$, the set
$\XX$ is a compact planar section in $\mathbb{R}^2$ (transverse to the Lorenz flow), and $\nu$ is an ergodic SRB measure. The hyperbolic properties of these maps are described in \cite{GP}, and due to the existence of a strong
stable foliation the dynamics in large part can be reduced to the one-dimensional Lorenz map discussed
in Section \ref{sec.nue}. In 
\cite{Licheng} it is shown that Condition (H2c) and Condition (H3) holds. Hence we have the following:
\begin{thm}\label{thm.lorenz2}
Suppose that $(f,\XX,\nu)$ is the family of Poincar\'e return maps associated to the geometric Lorenz flow
as described in \cite{Licheng}. Suppose that 
$\phi(x)=-\log(\dist(x,\tilde{x}))$, and $u_n:=u_n(u)$ is a sequence such that 
$\limsup_{n\to\infty}n\nu\{\phi(x)>u_n\}=\tau(u)<\infty.$
Then there exists $\alpha>0$ such that for $\nu$-a.e. $x\in\XX$ we have:
\[
\left| \nu\{M_n\leq u_n\}- G_{\sqrt{n}}(u)\}\right|\leq 
O(1)\exp\{-(\log n)^{\alpha}\}.
\]
\end{thm} 
The proof of this result follows directly from the proof Theorem \ref{thm.extreme2} with condition (H2b) replaced by (H2c).
We conjecture that this error estimate is sub-optimal and can be replaced an estimate of the form $O(n^{-\alpha})$
for some $\alpha>0$, see Section \ref{sec.lorenz}.


\section{General convergence estimates using blocking arguments and proof
of main theorems}\label{sec.general}
In this section we describe the theoretical basis for our choice of 
recurrence set $E_n$, and show how convergence to an EVD follows from the 
specific asymptotic properties of $\nu(E_n)$ along with the assumptions placed 
on the rate of mixing, and on the regularity of the invariant density. This
information will be specified by a blocking argument approach together with 
Propositions \ref{prop.blocking}-\ref{prop.extreme2} given below. We then show how the main theorems stated
in Section \ref{sec.statement.results} follow from these results. The propositions will be proved
in Section \ref{sec.proofs}. 

\subsection{The blocking argument and key estimates on convergence to an EVD}
\label{sec.props.errors}
We begin by giving an overview of the blocking algorithm as used 
in \cite{Collet, holland-nicol}. In the following, we give a precise 
quantification of the error rate in terms of the assumptions on the
correlation decay $\Theta(j)$, the decay of $\nu(E_n)$ and the regularity of 
$\nu$.  To state the propositions, we fix integers $p(n),q(n)>0$ and let $n=pq+r$ with $0\leq r<p$ (by Euclid's division algorithm). The blocking argument consists of choosing $q(n)$ blocks of length $p(n)$ with
$n\sim p(n)q(n)$. Between each of the blocks we take a gap of length $t=g(n)$. In particular
we choose $g(n)=o(p(n))$ and maintain the aymptotic $n\sim (p+t)q$.
The decay of correlations over the gap of length $t=g(n)$ allows us to consider successive blocks as approximately independent.  We suppose that 
$p,q\to\infty$ as $n\to\infty$. We let $u_n$ be the sequence with the 
property that $n\nu\{X_1>u_n\}\to\tau(u)$, for some function $\tau(u)$. 
At this stage we do not assume that $u_n$ has the representation $u/a_n+b_n$. 
We will assume that $\phi(x)$ has the representation 
$\phi(x)=\psi(\dist(x,\tilde{x}))$, for a monotonically decreasing function 
$\psi:[0,\infty)\to\mathbb{R}$. In particular we assume that $\psi(y)$ 
takes its maximum at $y=0$. We also write
$$M_{j,l}=\max\{X_{j+1},X_{j+2},\ldots, X_{j+l}\},\;\textrm{and}\;M_{0,l}=M_l.$$
For any integers $t,l,n$ we define the quantity
\begin{equation}
\gamma(n,t):=|\nu(X_1>u_n,M_{t,l}<u_n)-\nu(X_1>u_n)\nu(M_l<u_n)|.
\end{equation}
In the definition above we suppress the dependence on $l$ as it will not feature significantly in the estimates.
We have the following proposition:
\begin{prop}\label{prop.blocking}
Suppose that $f: \XX \to \XX$ is  ergodic with respect to an SRB
measure $\nu$. Then for $\nu$-a.e. $\tilde{x}\in\XX$, all $p,q$ such that $n=pq+r$, and $t<p$, we have
\begin{equation}
\left|\nu\{M_n\leq u_n\}-(1-p\nu\{X_1>u_n\})^q\right|\leq\mathcal{E}_n,
\end{equation}
where:
\begin{equation}\label{eq.mainerror}
\mathcal{E}_n =O(1)\left\{
\max\{qt,p\}\nu\{X_1\geq u_n\}+pq\gamma(t,n)+pq\sum_{j=2}^{p}\nu(X_1>u_n,X_j>u_n)\right\}.
\end{equation}
\end{prop}
The proof this proposition is purely probabilistic and the details
can be found in \cite{Collet, FF, holland-nicol}. In particular 
following \cite{Collet} and using asymptotic independence of blocks
of length $p$:
\begin{gather}
|\nu (M_n \le u_n ) - \nu(M_{q(p+t)}\le u_n)|\le \max\{qt,p\}\nu(X_1>u_n),\\
|\nu (M_{l(p+t)} \le u_n )- (1-p\nu(X_1>u_n))\nu(M_{(l-1)(p+t)}\le u_n)|\leq\tilde\Gamma_n,
\quad(l\in[1,q]),
\end{gather}
where
\[
\tilde\Gamma_n=
p\gamma(t,n)+2p\sum_{i=2}^{p}\nu(X_1>u_n, X_i>u_n)+t\nu(X_1>u_n).
\]
From this we deduce that 
\[
|\nu (M_n \le u_n ) - (1-p \nu (X_1 >u_n))^q|\le\max\{q,p/t\} \Gamma_n
\]
where
\[
\Gamma_n
=
\tilde\Gamma_n+t\nu (X_1>u_n).
\]
We remark that in the definition of $\tilde{\Gamma}_n$, the second summation 
can be bounded as follows:
\begin{equation}\label{eq.Gamma.split}
\begin{split}
\sum_{i=2}^{p}\nu(X_1>u_n, X_i>u_n) &\leq\sum_{i=2}^{t}\nu(X_1>u_n, X_i>u_n)+
p(\nu\{X_1>u_n\})^2\\
&\qquad+(p-t)\mathcal{C}_{\nu}\left(1_{\{X_1>u_n\}},1_{\{X_1>u_n\}}\circ 
f^t\right),
\end{split}
\end{equation}
where $\mathcal{C}_{\nu}(\varphi_1,\varphi_2)$ denotes the correlation 
between $\varphi_1,\varphi_2$ (with
respect to $\nu$). We use conditions (H2a)-(H2c) to estimate the sum on 
the right, and condition (H1) for the bounding the correlation term. 
To bound the correlation function we will use the same argument as
applied to bounding $\gamma(n,t)$.
We now use the dynamical assumptions to get refined asymptotics on each of the terms.
The first proposition gives an estimate for $\gamma(n,t)$.
For non-uniformly expanding maps an estimate is established
in \cite{holland-nicol}. We give a careful quantification in the non-uniformly
hyperbolic case, and in situations where the invariant density need not
lie in $L^p$.
\begin{prop}\label{prop.extreme1}
Suppose that $f: \XX \to \XX$ is  ergodic with respect to an SRB
measure $\nu$, and the local dimension $d_{\nu}$ exists for $\nu$-a.e. 
$x\in\XX$. We consider the following cases:
\begin{enumerate}
\item Suppose that $(f,\XX,\nu)$ is non-uniformly expanding and 
$\nu$ has a density $\rho\in L^{1+\delta}(m)$ for some 
$\delta>0$. Suppose that (H1) holds. Then there exists $\delta_1>0$ (depending on
$\delta$) such that 
\begin{equation}
\gamma(n,t)\leq O(1)\Theta(t)^{\delta_1}.
\end{equation}
\item Suppose that $(f,\XX,\nu)$ is non-uniformly expanding and 
there exists $\delta\in(0,1)$ such that for any Lebesgue measurable set $A$ 
we have $\nu(A)\leq O(1)\exp\{-|\log m(A)|^{\delta}\}$. If in addition
(H1) holds then there exists $\delta_1<1$ such that:
\begin{equation}
\gamma(n,t)\leq O(1)\exp\{-|\log\Theta(t)|^{\delta_1}\}.
\end{equation}
\item 
Suppose that $(f,\XX,\nu)$ is non-uniformly hyperbolic and 
conditions (H1s) and (H3) hold. Then there exist $\tilde\tau<1$ and
$\tilde\sigma>0$ such that
\begin{equation}
\gamma(n,t)\leq O(1)\max\{\Theta(t/2)^{\tilde\sigma},\tilde{\tau}^{t}\}.
\end{equation}
\end{enumerate}
\end{prop}
The following proposition gives a quantification of the first right hand term 
of equation \eqref{eq.Gamma.split} in terms of $\nu(E_n)$.
\begin{prop}\label{prop.extreme2}
Suppose that $f: \XX \to \XX$ is ergodic with respect to an SRB measure $\nu$, 
and the local dimension $d_{\nu}$ exists $\nu$-a.e. 
For given $\epsilon>0$ let $g(n)=\tilde{g}(n)^{1-\epsilon}$, where $\tilde{g}(n)$ is defined
in equation \eqref{eq:En}. Suppose that $\phi(x)=-\log(\dist(x,\tilde{x}))$, and $u_n$ is such that
$\limsup_{n\to\infty} n\nu\{\phi(x)>u_n\}<\infty$. We have the following.
\begin{enumerate}
\item Suppose that (H2a) holds. Then  there exists $\tilde\alpha>1$ such that for $\nu$-a.e. 
$\tilde{x}\in\mathcal{X}$:
\begin{equation}
\label{eq:dprime1}\sum_{j=2}^{g(n)}\nu(X_1>u_n,X_j>u_n)\leq C(\tilde{x}) 
\frac{g(n)}{n^{\tilde\alpha}}.
\end{equation}
\item 
Suppose that (H2b) holds for $\alpha>5$. Then there exists $\tilde\alpha>0$ such that for $\nu$-a.e. 
$\tilde{x}\in\mathcal{X}$:
\begin{equation}
\label{eq:dprime2}\sum_{j=2}^{g(n)}\nu(X_1>u_n,X_j>u_n)\leq C(\tilde{x}) 
\frac{g(n)}{n(\log n)^{\tilde\alpha}}.
\end{equation}
\item
Suppose that (H2c) holds. Then there exists $\tilde\alpha>0$ such that for $\nu$-a.e. 
$\tilde{x}\in\mathcal{X}$:
\begin{equation}
\label{eq:dprime3}\sum_{j=2}^{g(n)}\nu(X_1>u_n,X_j>u_n)\leq C(\tilde{x}) 
g(n)\exp\{-(\log n)^{\tilde\alpha}\}. 
\end{equation}
\end{enumerate}
In each case the constant $C$ depends on $\tilde{x}$, and $\tilde\alpha$ 
depends on $\alpha$ appearing in (H2a)-(H2c). 
\end{prop}
In the proof Proposition \ref{prop.extreme2} above, we distinguish 
between non-uniformly expanding systems and non-uniformly hyperbolic systems 
and optimize the constant $\tilde\alpha$ in each case. In particular the precise bound
on the optimal value of $\tilde\alpha$ will depend on $\alpha$. The constant
$\tilde\alpha$ is used for bounding the rate of convergence to an EVD.

\subsection{Proof of main theorems}\label{sec.thm.proofs}
Using the propositions stated in Section \ref{sec.props.errors} we show
how the main theorems stated in Section \ref{sec.statement.results} follow. We will defer
the proof of Theorem \ref{thm.stretched} to Section \ref{sec.proofstretched}
as the proof requires the actual estimation of both $\nu(E_n)$ and the
regularity of the invariant density.
 
\paragraph{Proof of Theorem \ref{thm.extreme}.}
We will take $q=p=\sqrt{n}$ in Proposition \ref{prop.blocking}
and take $t=g(n)=(\log n)^{\tilde\gamma}$ for some $\tilde\gamma<\gamma$ (maintaining $\beta\tilde\gamma>1$). 
In Case 1 of Theorem \ref{thm.extreme}, 
Proposition \ref{prop.extreme1} implies that (for some $c>0$):
$$\gamma(n,t)\leq O(1)\Theta((\log n)^{\tilde\gamma})^{\delta_1}
\leq O(1)\exp\{-c\delta_1(\log n)^{\beta\tilde\gamma}\}.$$
This term goes to zero at a superpolynomial rate provided 
$\beta\tilde\gamma>1$. By choice of 
$u_n$, $\nu\{X_1>u_n\}\leq O(1)n^{-1}$, and therefore the dominating term
comes from Case 2 of Proposition \ref{prop.extreme2}. If $\alpha$
is the constant in (H2b), then the proof of Proposition \ref{prop.extreme2}
gives $\tilde\alpha<\alpha-1$. Since we must $\tilde\gamma>\beta^{-1}$ the bound on $\mathcal{E}_n$
(and hence that in equation \eqref{eq.thm1.h2b}) follows. The Case 2 of Theorem \ref{thm.extreme} is similar. 
In this case the dominating term comes from Case 3 of Proposition \ref{prop.extreme2}
with $\tilde\alpha$ chosen to be any constant less than $\alpha-1$.
\paragraph{Proof of Theorem \ref{thm.lorenz1}.}
The proof is straightforward in light of the proof of 
Theorem \ref{thm.extreme}, Case 2. For one-dimensional Lorenz maps it is proved 
that (H2c) holds for some $\alpha\in(0,1)$ and $\gamma=5$, see \cite{GHN}.
Moreover the invariant density $\rho$ is of bounded variation type, 
and hence in $L^{\infty}$. Hence this establishes: 
$$\left|\nu\{M_n\leq u+\log n\}-G_{\sqrt{n}}(u)\right|\leq O(1)\exp\{-(\log n)^{\alpha_1}\},$$
for some $\alpha_1>0$.
Using the regularity of the invariant density $\rho$, it follows that for $\nu$-a.e. $\tilde{x}\in\XX$,
we have 
$$\tau_n(u)=n\nu\{\phi(x)>u+\log n\}=C(\tilde{x})e^{-u}+O(1/n),$$
and hence $G_{\sqrt{n}}(u)\to G(u)$ up to an error of order $1/\sqrt{n}$ (which gives an
insignificant contribution). Hence we get the required convergence to the Gumbel distribution
as stated in the theorem.

\paragraph{Proof of Theorem \ref{thm.partial}.}
Proof of convergence to an EVD (without an error bound) was established in
\cite{Gupta}. The proof uses a blocking argument approach, and in fact
Proposition \ref{prop.blocking} applies to this system. For partially
hyperbolic systems, the quantitative recurrence statistics assumptions
are phrased in terms of $E^{X}_n$. If we let
$$E_{n}:=\left\{(x,\theta)\in X\times Y:\,d_X(f^j(x,\theta),(x,\theta))
<n^{-1/d},\;\textrm{some}\;j\in[1,\tilde{g}(n)]\right\},$$
then by \cite[Proposition 3.4]{Gupta}, it is shown for some $C>0$ that
$\nu(E_{n})\leq C\nu_X(E^{X}_{n'})$, with $n'=n^{d_X/d}$. Consider now Case 1 
of Theorem 
\ref{thm.partial}. We will take $q=p=\sqrt{n}$ in Proposition 
\ref{prop.blocking} and take $t=g(n)=(\log n)^{\tilde\gamma}$ for some $\tilde\gamma<\gamma$. As in the
proof of Theorem \ref{thm.extreme}, $\gamma(n,t)$ tends to zero at
a superpolynomial rate provided we choose $\tilde\gamma$ so that $\beta\tilde\gamma>1$. If $\alpha$
is the constant in (H2b), then Proposition \ref{prop.extreme2}
implies that we can choose $\tilde\alpha$ arbitrarily close to $\alpha-1$. In Case 2, we take $q=p=\sqrt{n}$ in Proposition 
\ref{prop.blocking} but this time take $t=g(n)=n^{\kappa}$ for some $\kappa<\gamma$.
Proposition \ref{prop.extreme1} implies that
$$\gamma(n,t)\leq\Theta(n^{\kappa})^{\delta_1}
\leq O(1)n^{-\zeta\kappa\delta_1}.$$
If we are to have $\mathcal{E}_n\to 0$ then we require $pq\gamma(n,t)=o(1)$,
and hence $\kappa>(\zeta\delta_1)^{-1}$.  From Case 1 of
Proposition \ref{prop.extreme2} we get an error contribution of the order $n^{\alpha-\kappa}$
using (H2a). Combining these errors gives the contribution as stated in the theorem. Notice
that we require $(\zeta\delta_1)^{-1}<\alpha$ if we are to have $\mathcal{E}_n\to 0$.
In each case we have $\tau_n(u)\to\tau(u)=C(\tilde{x})e^{-u}$, and hence we get
convergence to the Gumbel distribution. However, the precise error rate involved 
depends on refined properties 
of the invariant density (such as having H\"older continuity).

\paragraph{Proof of Theorem \ref{thm.alvesviana}.}
The proof of this theorem combines that of Theorem \ref{thm.partial}
and Lemma \ref{lem.density} in Section \ref{sec.proofstretched}. The latter result is required since we do not
know apriori that the invariant density $\nu$ belongs to some $L^p$, 
for some $p>1$. For the Alves-Viana map, it is shown in \cite{gouezel}
that $\exists\theta_0,\beta<1$ such that $\Theta(n)\leq O(\theta^{n^{\beta}}_0)$.
Hence by Lemma \ref{lem.density}, there exists $\hat\beta>0$ such 
that for any measurable set $A\subset\XX$:
$$\nu(A)=O\left(\exp\{-c|\log m(A)|^{\hat\beta}\}\right).$$
We will take $q=p=\sqrt{n}$ in Proposition \ref{prop.blocking}
and take $t=g(n)=(\log n)^{\tilde\gamma}$. Proposition \ref{prop.extreme1} implies that
$$\gamma(n,t)\leq
O(1)\exp\{-|\log\Theta(t)|^{\delta_1}\}
\leq O(1)\exp\{-c\delta_1(\log n)^{\delta_1\tilde\gamma\beta}\}.$$
The constant $\delta_1$ depends on $\beta$.
If we take $\tilde\gamma$ so that $\tilde\gamma\beta\delta_1>1$, then $\gamma(n,t)$
goes to zero at a superpolynomial rate. Let us now estimate
the contribution to $\mathcal{E}_n$ coming from $\nu(E_n)$.
Since the base transformation of the Alves-Viana map is 
a uniformly hyperbolic Markov map it follows that
$\nu_X(E^{X}_n)\leq Cn^{-1}\tilde{g}(n)$, see for example \cite{holland-nicol}
for a similar calculation. Hence as in the proof of Theorem \ref{thm.partial}
we have $\nu(E_n)\leq Cn^{-1/2}$. By the choice of $g(n)=(\log n)^{\tilde\gamma}$ the
error estimate stated in Theorem \ref{thm.alvesviana} 
follows from Proposition \ref{prop.extreme2}. 

\paragraph{Proof of Theorem \ref{thm.extreme2}.}
We will take $q=p=\sqrt{n}$ in Proposition \ref{prop.blocking}
and take $t=g(n)=(\log n)^{\tilde\gamma}$ for some $\tilde\gamma>1$. 
Consider Case 1 of Theorem \ref{thm.extreme2}. Applying Case 3 of
Proposition \ref{prop.extreme1} we obtain
$$\gamma(n,t)\leq \max\{\Theta(t/2)^{\tilde\sigma},\tilde{\tau}^{t}\}
\leq O(1)\exp\{-c\delta_1(\log n)^{\beta\gamma}\}.$$
This term goes to zero at a superpolynomial rate provided 
$\beta\tilde\gamma>1$. As in the proof of Theorem \ref{thm.extreme}
the dominating term comes from Case 2 of Proposition \ref{prop.extreme2}. 
If $\alpha$ is the constant in (H2b), then the proof of Proposition 
\ref{prop.extreme2} (in the non-uniformly hyperbolic case) 
implies that we can take $\tilde\alpha$ arbitrarily close to $\alpha-5$. 
From this, the bound on $\mathcal{E}_n$ follows provided $\alpha$ is sufficiently large. 
Proof of Case 2 in Theorem \ref{thm.extreme2} follows similarly, and in particular 
see Remark \ref{rmk.h2a.hyp}.  
As in the proof of Theorem \ref{thm.partial} we get
convergence to the Gumbel distribution.
\paragraph{Proof of Theorem \ref{thm.lorenz2}.}
We will take $q=p=\sqrt{n}$ in Proposition \ref{prop.blocking}
and take $t=g(n)=(\log n)^{\tilde\gamma}$ for some $\tilde\gamma>1$. 
For two-dimensional Lorenz maps it is proved 
that (H2c) holds for some $\alpha\in(0,1)$ and $\gamma=2$, see \cite{Licheng}.
Moreover it is shown there that condition (H1s) holds with
$\Theta(n)\leq O(\theta^{n}_0)$ for some $\theta_0<1$, and condition (H3) holds
for some $\sigma>0$. by Proposition \ref{prop.extreme1}, it follows that for all $\epsilon>0$:
$$\gamma(n,t)\leq \max\{\Theta(t/2)^{\tilde\sigma},\tilde{\tau}^{t}\}
\leq O(1)\exp\{-(\log n)^{2-\epsilon}\},$$
and this term goes to zero at a superpolynomial rate. Using Case 3 of
Proposition \ref{prop.extreme2}, we find that there exists $\alpha_1>0$
such that $\mathcal{E}_n\leq O(1)\exp\{-(\log n)^{\alpha_1}\}$.

\section{Further remarks on quantitative recurrence, including flows and quasiperiodic 
systems}\label{sec.furtherdiscussion}
We have so far discussed the role of the set $E_n$ in estimating
the rate of convergence to an EVD for a broad class of hyperbolic dynamical
systems. In this section we give a precise link between $E_n$
and the notion of having \emph{short return times}. We then discuss
quantitative recurrence statistics for flows and also for quasi-periodic systems.

\subsection{On the link between $E_n$ and short return times}\label{sec.SRT}
For hyperbolic systems, we phrased our dynamical assumptions in terms
of the set $E_n$ and the asymptotics of $\nu(E_n).$ In this section,
we give a brief note on how these assumptions link to the 
\emph{short return time} conditions as presented in  
\cite{CC,HNPV,Haydn-Wasilewska,holland-nicol}. Motivated from
these references we give the following definition:
\begin{defn}\label{def.srt}
We say that the \emph{short return time} (SRT) condition holds
for $(f,\XX,\nu)$ if there is a set $\Lambda\subset\XX$,
constants $\gamma>1,C,\alpha,s>0$ with $\nu(\XX\setminus\Lambda)\leq Cr^s$,
and for all $\tilde{x}\in\Lambda:$
\begin{equation}
\nu\left(B(\tilde{x},r)\cap f^{-k}B(\tilde{x},r)\right)\leq
O(1)r^{\alpha}\nu\left(B(\tilde{x},r)\right),
\end{equation}
for all $k=1,\ldots, |\log r|^{\gamma}$.
\end{defn}
For hyperbolic systems such as the Lozi map and billiards,
the SRT condition is shown to hold, see \cite{HNPV}. In the other direction
it is shown in \cite{GHN} that for $\tilde{g}(n)=(\log n)^{\gamma}$ we
have $\nu(E_n)\leq O(n^{-\alpha})$. For systems admitting rank one Young
towers with exponential decay of correlations, 
it is shown via \cite[Proposition 4.1]{CC} that the above SRT condition holds.
In particular their hypotheses capture the H\'enon map application. To
relate the SRT condition to the recurrence set $E_n$, we have the following result:
\begin{prop}\label{prop.srt}
Suppose that the SRT condition holds with constants specified
in Definition \ref{def.srt}. Then there exist $\gamma'>1$ and $\alpha'>0$ such
that $\tilde{g}(n)=(\log n)^{\gamma'}$ implies  $\nu(E_n)\leq O(n^{-\alpha'})$
\end{prop}
For systems with polynomial decay of correlations similar SRT conditions are formulated
in \cite{Haydn-Wasilewska}, and in their case we expect $\nu(E_n)$ to have
a logarithmic asymptotics.
Before proving this proposition we state the following estimate
that quantifies how measures scale on small balls. 
\begin{lemma}\label{lem.Bes}
Let $\mathcal{F}(\lambda,r)=\{x\in\XX:\nu(B(x,2r))>\lambda\nu(B(x,r))\}.$
Then there exists $C>0$ independent of $r$ such that 
$$\nu(\mathcal{F}(\lambda,r))\leq C\lambda^{-1}.$$
\end{lemma}
In the special case where $\lambda=r^{-s}$ this lemma is proved in 
\cite[Lemma A2]{CC} using the Besicovitch covering theorem \cite{mattila}. 
As can be inspected from their proof, the argument does not depend on the explicit
form of $\lambda$. See also \cite{Haydn-Wasilewska} in the case where 
$\lambda=|\log r|^{-s}$. However the scaling factor in the ball radii 
is important in the proof. To find the measure of the set
$$\widetilde{\mathcal{F}}(\lambda,r,c)
=\{x\in\XX:\nu(B(x,cr)>\lambda\nu(B(x,r))\},$$
(for $c>2$), then we would need to iteratively apply Lemma \ref{lem.Bes}. 
We will do this in the proof of Proposition \ref{prop.extreme2}.

\noindent {\it Proof of Proposition \ref{prop.srt}:}
Consider the set $F_j(r)=\{x:\dist(x,f^j(x))\leq r\}$. Then we have:
\begin{eqnarray*}
F_j(r)\cap B_{r}(x) & \subset & \{y\in B_{r}(x):\dist(y,f^{j}(y))\leq r\} \\
 & \subset & \{y\in B_{r}(x): f^{j}(y)\in  B_{2r}(x)\} \\
 & \subset & B_{2r}(x)\cap f^{-j}(B_{2r}(x)).
\end{eqnarray*}
Hence, by Lemma \ref{lem.Bes} and the SRT condition, there exist $\alpha_1>0$
and $s_1>0$ such that for all $j\leq |\log r|^{\gamma}$ we have
\begin{equation}\label{eq.bes1}
\nu(F_j(r)\cap B_{r}(x))\leq r^{\alpha}\nu(B(x,2r))\leq r^{\alpha-s_1}\nu(B(x,r)),
\end{equation}
provided $x\not\in\Lambda'$, where $\Lambda'=\XX\setminus\Lambda$ 
is such that $\nu(\Lambda')\leq C\max\{r^{s_1},r^s\}$. Using equation
\eqref{eq.bes1} we can now estimate 
$\nu(F_j(r))$ by taking a a cover of $F_{j}(r)$ using disjoint balls and 
applying the Besicovitch Covering Lemma. For some
$\alpha_2>0$ we obtain $\nu(F_j(r))\leq C(n^{-\alpha_2})$. The corresponding
estimate for $\nu(E_n)$ follows by setting $r=1/n$ and summing over $j\in[1,(\log n)^{\gamma}]$.
\hfill $\square$

\subsection{Convergence to EVD for suspension flows}\label{sec.suspension}

Assume that $(f,\XX,\nu)$ is a measure preserving system and
that $h\in L^{1}(\nu)$ is a positive roof function.
Consider the suspension space
\[
\XX^{h}=\{(x,u)\in \XX\times\mR \mid 0\leq u\le h(x)\}\,/\sim,
\qquad (x,h(x))\sim(f(x),0).
\]
We denote the suspension (semi) flow by
\[
g_t : \XX^h \to \XX^h, \quad
g_t(x,u)=(x,u+t)/\sim.
\]
On $\XX^h$ introduce the flow-invariant probability measure $\nu^h$ given by
$\nu \times m/ \bar h$ and $\bar h=\int_\XX h d \nu$. For flows, the natural
definition of the recurrence set is:
\begin{equation}\label{eq:En-flow}
    E_T(\gamma):=\left\{ x\in \XX^h: \dist(x,g_t x)\le\frac{1}{T^{1/d}}
    \,\textrm{for some}\;
t\in[\delta_0,T^{\gamma}]\right\}.
    \end{equation}
The choice of $\delta_0>0$ is arbitrary but it will be convenient to take 
$\delta_0\leq\inf h$, where we assume $\inf h>0$.  Consider a (measurable) observation $\phi:\XX^h\to\mR$ and define $M_T:\XX^h\to
\mR$ by
\begin{equation}\label{eq.Phi}
M_{T}(x):=\max \{ \phi(g_t(x)) \mid 0\le t < T \}.
\end{equation}
In \cite{HNT}, it is shown that if the base transformation $f$ satisfies 
convergence to an EVD, then for suitable scaling constants $a_{T},b_{T}$, 
the process $a_T(M_T-b_T)$ also converges in law to one of standard EVD types. 
Suppose $u_T=u/a_{T}+b_T$ is a sequence such that
$$\lim_{T\to\infty}T\nu^{h}\{\phi(x)\geq u_T\}=\tau(u),$$
and suppose for $\gamma>0$ and $\alpha>0$ 
we have $\nu^h(E_{T}(\gamma))\leq CT^{-\alpha}$. Then under suitable
hypothesis on the rate of mixing, and on the regularity of $\nu$ and $\phi$,   
we can conjecture that there exists $\alpha'>0$ such that
\begin{equation}
\left|\nu\{M_t\leq u_T\}-G_{\sqrt{T}}(u)\right|\leq \frac{C}{T^{\alpha'}}.
\end{equation}
In Section \ref{sec.lorenz} we will study the recurrence set $E_T$ 
for the Lorenz flow.
To relate $E_{T}(\gamma)$ to that of $E_n(\gamma)$ we observe the following:
\begin{lemma}\label{lem.lift}
Suppose that $(g_t,\XX^h,\nu^h)$ is a suspension flow over
$(f,\XX,\nu)$ with roof function $h\in[h_{m},h_M]\subset(0,\infty)$.
Then there exist $\gamma,\gamma'>0$ and $\alpha,\alpha'>0$ such that
$$\nu^h(E_{T}(\gamma))\leq O(T^{-\alpha})\Leftrightarrow \nu(E_{n}(\gamma'))\leq O(n^{-\alpha'}).$$
\end{lemma}
The proof of this lemma is straightforward if we observe that for all
$n\geq 0$ we have $T\in[nh_m,(n+1)h_M]$. In this case we can in fact take
$\gamma'=\gamma$ and $\alpha'=\alpha$. In general
the constants $\gamma',\alpha'$  will depend on the regularity of $h$
if (for example) we allow $\sup h(x)=\infty$.


\subsection{Quasi-periodic systems.}\label{sec.quasiperiodic}
For quasi-periodic systems it is known that non-standard limits exist
for the distribution of the return times, see \cite{Coelho, Coelho-F}. In particular 
quasi-periodic systems are not mixing, and hence condition (H1) is not valid. Moreover, 
we observe that the measure of $E_n(\gamma)$ 
abruptly changes from positive to zero as $n$ is increased
(for fixed $\gamma<1$). This is unlike what is observed for hyperbolic systems. 
To see this intuitively, recall that the dynamical properties of quasi-periodic systems can be described in terms of their rotation number. 
If the rotation number is irrational then there are no periodic orbits, and therefore if $E_n(\gamma)\neq\emptyset$ 
it will not contain periodic points. For $\gamma$ sufficiently small, we find that the set 
$E_n(\gamma)$ (for all sufficiently large $n$) is empty or quite meagre with zero measure. 
If $\gamma$ is chosen sufficiently large, or if the system has rational rotation number then 
we find that $\nu(E_n)$ is uniformly bounded away from zero (for all $n$). 
From a point of view of numerical diagnostic tests for convergence to an EVD, we might deduce 
that an abrupt change in $\nu(E_n(\gamma))$ indicates that the statistics of extremes are 
governed by a non-standard limit law.

\paragraph{Circle rotation maps.}
As a case study, consider the circle rotation map $f(x)=x+\theta$ on $S^1=[0,1]/(0\sim 1)$,
and $\theta\in[0,1]$. For fixed $\gamma<1$, we show that for typical $\theta\in[0,1]$ the
measure $\nu(E_n(\gamma))$ abruptly drops 
to zero as $n$ increases. The discussion below is also applicable to minimal circle homeomorphisms 
with rotation number $\theta$. Let  $\tau_r(x):=\inf\{k\geq 1:f^k(x)\in B(x,r)\}$. 
For circle rotations, unique ergodicity allows us to obtain bounds on 
$\nu(E_n(\gamma))$ via the statistics of  $\tau_r(x)$, at least for typical 
rotation numbers. In \cite{Kim} it is shown that for 
\emph{all} $x\in S^1$:
\begin{equation}
\liminf_{r\to 0}\frac{\log\tau_r(x) }{-\log r}=\frac{1}{\eta}
\quad\mbox{and}\quad
\limsup_{r\to 0}\frac{\log\tau_r(x) }{-\log r}=1,
\end{equation}
where $\eta=\sup\{\beta:\liminf_{k\to\infty}k^{\beta}\|k\theta\|=0\}.$
Here $\|k\theta\|$ denotes the nearest integer to $k\theta$. For
Lebesgue almost all $\theta\in[0,1]$, $\eta=1$. Liouville numbers (of
measure zero) correspond to $\eta=\infty$, an example being
$\theta=\sum_{k\ge 1}10^{-k!}.$ A consequence is the following result:
\begin{prop}\label{prop:quasi} Suppose $f:S^1\to S^1$ is a circle rotation map. For
Lebesgue almost all $\theta\in[0,1]$ and all $\gamma<1$, there exists $N>0$ such that
for any $n>N$, we have $\nu(E_n(\gamma))=0$.
\end{prop}
We remark that the conclusion of Proposition \ref{prop:quasi} also applies to more general 
circle homeomorphisms such as the Arnold Family. See Section \ref{sec.num.quasiperiodic}.


\section{A numerical procedure to estimate $\nu(E_n(\gamma))$}
\label{sec.numerics}
For selected dynamical systems we 
compare numerical estimates on the decay of $\nu(E_n(\gamma))$ to our analytic
results, and study systems for which there are conjectural power law bounds on 
$\nu(E_n(\gamma))$. Examples include the H\'enon Map (for the classical parameter values), 
Axiom-A diffeomorphisms and the Lorenz-63 flow. 

We now outline the numerical approach. Consider a dynamical system $(f,\XX, \nu)$ as in the previous sections. For a $\nu$-measurable 
set $A$ we define
\begin{equation}\label{eq.nu-est}
\nu_{\rm est}(A; N, x)
:=
\frac{1}{N} \sum_{j=0}^{N-1} 1_A(f^j(x)).
\end{equation}
Birkhoff's Ergodic Theorem implies that for $\nu$-a.e.\ $x\in\XX$
\[
\nu(A) = \lim_{N\to\infty} \nu_{\rm est}(A; N, x).
\]

Assume now that $(f,\XX, \nu)$ has decay of correlations for Lipschitz 
continuous functions with rate function $\Theta(n)\to 0$ as in (H1) or (H1s).  
If the correlation decay is fast enough
(e.g. $\Theta(n) = O(n^{-(2+\varepsilon)})$ for some $\varepsilon>0$), 
then we expect the Central Limit Theorem (CLT) to hold 
for the invariant measures, see \cite{Young1}. When the CLT applies, 
the sample estimates $\nu_{\rm est}(\cdot; N, x) $ are approximately normal 
with mean $\nu(\cdot)$ and standard deviation $\frac{\sigma^2}{\sqrt{N}}$, 
where $\sigma$ is a constant that depends on the decay of correlations. 

Given a set $A$ in $\XX$, we follow a double sampling procedure to 
estimate its measure. For a fixed $x_0$ in the support of $\nu$, we 
consider its following $N$ iterations and we estimate 
$\nu(A)$ by $\hat{\nu}_0 = \nu_{\rm est}(A,N,x_0)$. 
We repeat this for $M$ different starting points $x_0,\cdots,x_{M-1}$ 
obtaining a set $\hat{\nu}_0, \hat{\nu}_1, \dots, \hat{\nu}_{M-1}$
of $M$ estimations. Finally, we estimate $\nu(A)$ by the sample mean
\[
\hat{\nu} = \frac{1}{M}\sum_{m=0}^{M-1} \hat{\nu}_m
\]
and use the sample standard deviation
\[
s_{\nu}^2 = \frac{1}{M-1}\sum_{m=0}^{M-1} (\hat{\nu}_m - \hat{\nu})^2
\]
to estimate the uncertainty in the approximation of $\nu(A)$.

In the following subsections we apply this procedure to estimate 
$\nu(E_n)$ in order to check the quantitative recurrence conditions, in particular condition (H2a). 
In the cases with slow 
decay of correlations we expect the uncertainty of the estimation 
to behave badly. Hence, the procedure also provides an indirect 
check about condition (H1). 


When applying this procedure we need an initial set 
$x_0,\cdots,x_{M-1}$ of starting points lying in the 
support of $\nu$. Typically we generate $x_0$ by applying a 
transient of a few thousand iterations to a random point in $\XX$. 
For the rest of the points $x_1,\dots,x_{M-1}$, we take $x_i = f^{N+1}(x_{i-1})$ 
for the sake of efficiency. 
In the examples below we used $M=20$ samples and $N=10^4$ points unless 
specified otherwise. The figures display the confidence intervals 
$\hat{\nu}\pm 1.96\times s_\nu$ as a function of $n$. If the estimates 
of $\nu(E_n)$ fit well to a straight line on a log-log
plot this suggests that $\nu(E_n)$ decays as a power law in $n$.


\subsection{Quantitative recurrence rates for non-uniformly expanding systems}
\label{sec.numeric.nue}
\DeclareGraphicsExtensions{.eps}
\begin{figure}[t]
\centering
\psfrag{n}{$n$}
\psfrag{meas}{$\nu(E_n(\gamma))$}
\includegraphics[width=0.49\textwidth]{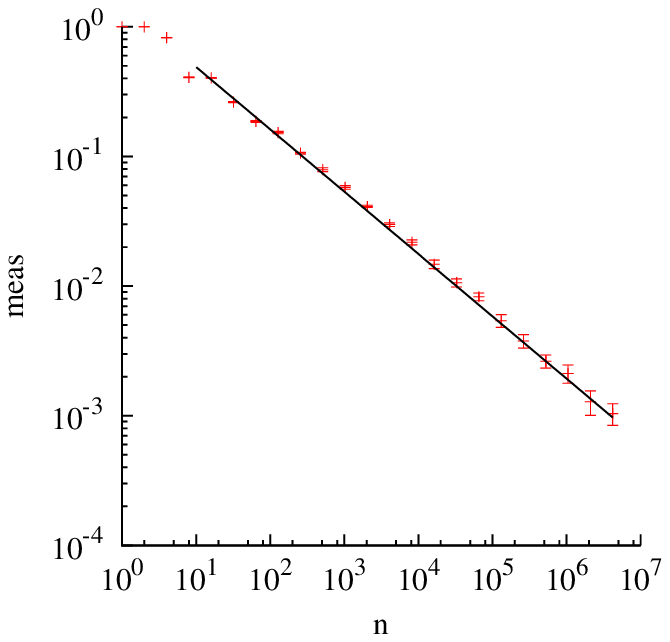}
\includegraphics[width=0.49\textwidth]{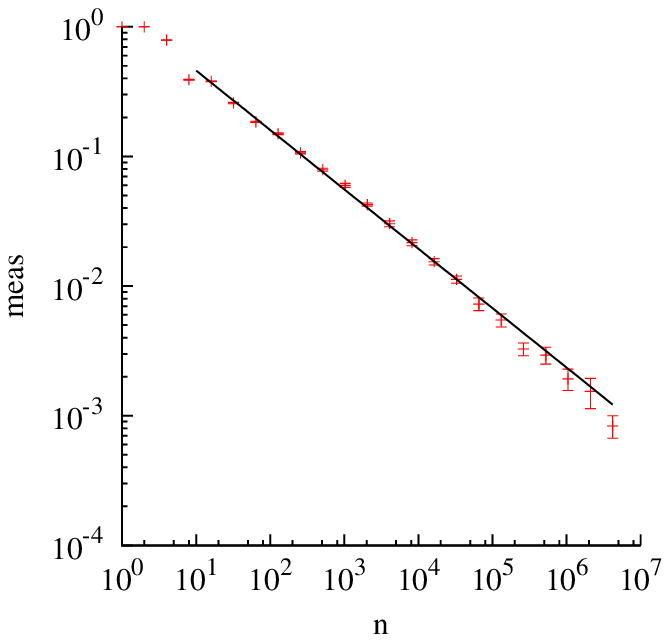}
\caption{
Numerical estimate of $\nu(E_n(\gamma))$ for $\gamma=0.5$ versus $n$ for the
perturbed doubling map \eqref{eq.example-doubling} for $\epsilon=0.01$ (left)
and $\epsilon=0.1$ (right). 
Note that the scale on both axes is logarithmic. The estimates of $\nu(E_n(\gamma))$ 
fit well to a straight line, which suggests that $\nu(E_n(\gamma))$ decays as a 
power law in $n$.}
\label{fig.doubling}
\end{figure}

\begin{figure}[t]
\centering
\psfrag{n}{$n$}
\psfrag{meas}{$\nu(E_n(\gamma))$}
\includegraphics[width=0.49\textwidth]{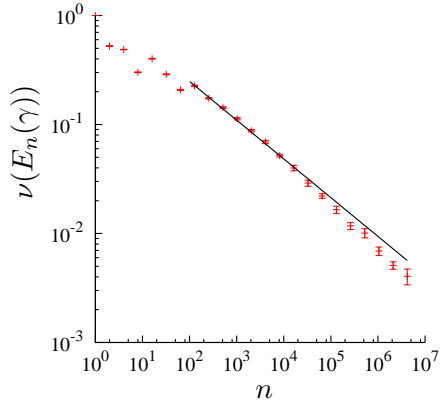}
\caption{
As Figure \ref{fig.doubling}, but for the quadratic family \eqref{eq.logistic}
with $a=3.9$ and $\gamma=0.5$. 
} 
\label{fig.logistic}
\end{figure}

\begin{figure}[t]
\centering
\psfrag{n}{$n$}
\psfrag{meas}{$\nu(E_n(\gamma))$}
\includegraphics[width=0.49\textwidth]{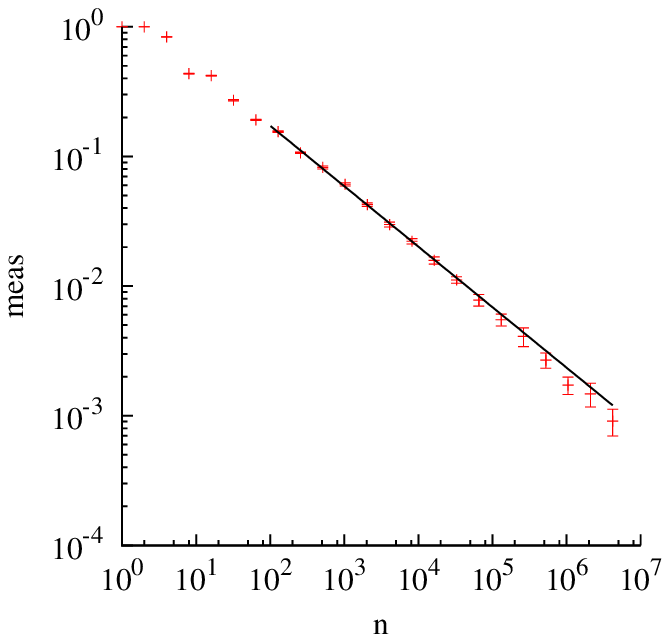}
\includegraphics[width=0.49\textwidth]{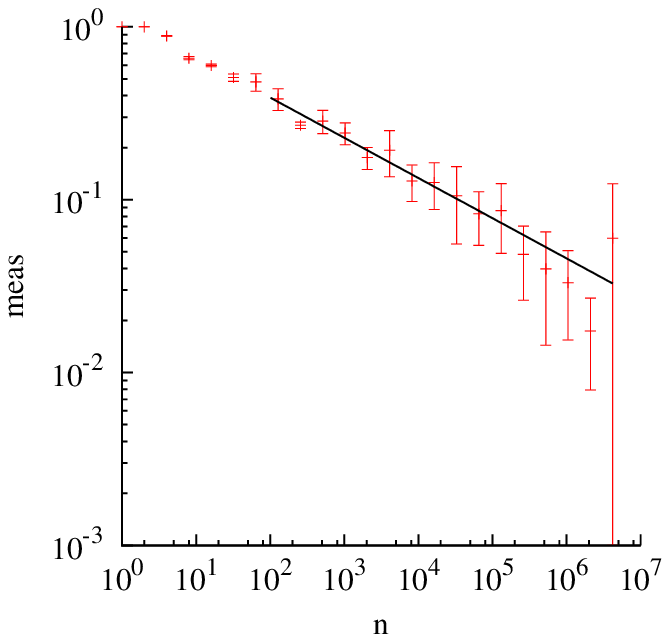}
\caption{
As Figure \ref{fig.doubling}, but for the intermittency map \eqref{eq.intermittency}
with $b=0.1$ (left) and $b=0.7$ (right). In both cases $\gamma=0.5$. 
} 
\label{fig.intermittency}
\end{figure}

\begin{figure}[t]
\centering
\psfrag{n}{$n$}
\psfrag{meas}{$\nu(E_n(\gamma))$}
\includegraphics[width=0.49\textwidth]{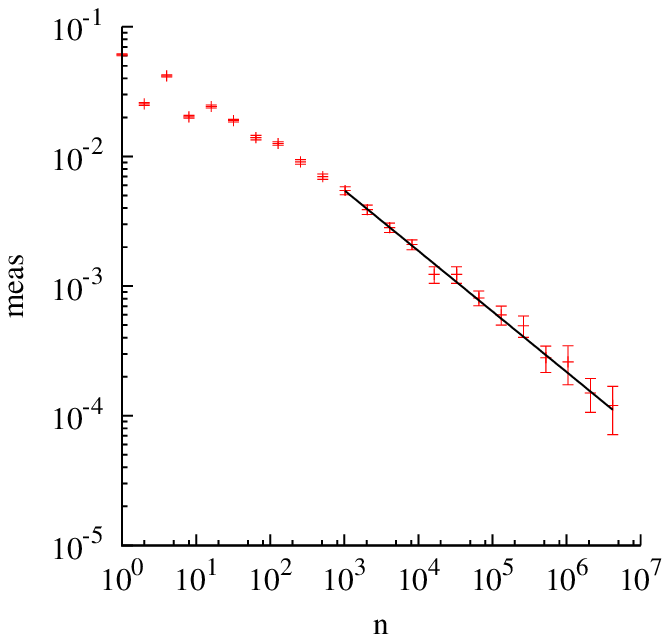}
\caption{
As Figure \ref{fig.doubling}, but for the Alves-Viana map \eqref{eq.av} with
with $a=1.9$ and $\epsilon=0.01$. We take $\gamma=0.5$ 
} 
\label{fig.av}
\end{figure}

We apply the procedure described above to certain (non-uniformly) expanding dynamical systems 
which include the quadratic family of maps, intermittency maps and the Alves-Viana map.
 We estimate $\nu(E_n(\gamma))$ and contrast this estimate to the theoretical results.

\paragraph{Uniformly expanding maps.}
Consider the perturbed doubling map on the circle $S^1=[0,1]/(0\sim 1)$ given by
\begin{equation}\label{eq.example-doubling}
f(x)=2x+\epsilon\sin(2\pi x) \, \mod 1,
\end{equation}
where $\epsilon>0$ is sufficiently small. This is a uniformly expanding system
and it is known that (H1) and (H2a) hold. For $\epsilon=0$ it is known that
$\nu(E_n(\gamma))\leq n^{-1+\gamma}$, see \cite[Section 3.1]{holland-nicol}.    
For $\epsilon\neq 0$, and $\gamma=1/2$, we
observe numerically that $\nu(E_n(\gamma))$ decays as a power law in $n$:
for $\epsilon=0.01$ we have $\nu(E_n(\gamma))=O(n^{-0.48})$ and for
$\epsilon=0.1$ we have $\nu(E_n(\gamma))=O(n^{-0.46})$, see Figure
\ref{fig.doubling}. Using the blocking argument of Section \ref{sec.general}, we can balance
this estimate against that given by exponential decay of correlations over the time
scale $\tilde{g}(n)$. Thus we might choose instead $\tilde{g}(n)=(\log n)^{\gamma}$
to improve the overall estimate. Convergence to EVD would be of the order 
$O(n^{-1/2+\epsilon_1})$, (for any $\epsilon_1>0$). 


\paragraph{The quadratic family.}
This quadratic family of maps $(f,\XX,\nu)$, with $\XX=[0,1]$, is given by
\begin{equation}\label{eq.logistic}
f(x)=ax(1-x),\quad a\in[0,4]
\end{equation}
There is a positive measure subset $\Omega$ of parameter values close to $a=4$, where
$(f,I,\nu)$ admits a Young tower with exponential return time asymptotics. Numerically we  
explore $\nu(E_n(\gamma))$ for the parameter value $a=3.9$ (where the map appears to exhibit 
chaos), and choose $\gamma=0.5$. We see in Figure 
\ref{fig.logistic} that $\nu(E_n(\gamma))$ decays as a power law with 
$\nu(E_n(\gamma))=O(n^{-0.36})$. This suggests that we get a fast convergence rate
to the EVD relative to the theoretical results. As discussed \cite[Section 3.3]{holland-nicol}
the best that can achieved is the existence of an $\alpha>0$ such that $\tilde{g}(n)=(\log n)^5$
implies $\nu(E_n)\leq n^{-\alpha}$. The above estimate suggests we can take $\alpha$ much 
greater than 0.36, since by use of exponential decay of correlations we can use the time 
scale $\tilde{g}(n)=(\log n)^5$.  


\paragraph{Non-uniformly expanding intermittency maps.}

Consider the interval map $f:I\to I$, with $I=[0,1]$ and $b>0$, given by
\begin{equation}\label{eq.intermittency}
f(x)
=
\begin{cases}
x(1+(2x)^b) & \mbox{for } 0 \leq x < \tfrac{1}{2}, \\
2x-1 & \mbox{for } \tfrac{1}{2} \leq x \leq 1.
\end{cases}
\end{equation}
For $b\in(0,1)$ this system $(f,I,\nu)$ admits a Young tower with polynomial return time 
asymptotics. For $b\geq 1$,  the invariant (physical) measure of $f$ is no longer 
absolutely continuous with respect to Lebesgue: it is the Dirac measure at $\{0\}$.
For $b<1/20$, analytic estimates on the convergence rate to an EVD where obtained in 
\cite[Section 3.2]{holland-nicol}. We indicate here that those bounds might extend
to a wider parameter range. For $b=0.1$ we have $\nu(E_n(\gamma))=O(n^{-0.46})$,
and for $b=0.7$ we have $\nu(E_n(\gamma))=O(n^{-0.23})$. The power law indeed weakens
as $b\to 1$ (for fixed $\gamma=1/2$), see Figure~\ref{fig.intermittency}. By the blocking 
argument of Section \ref{sec.general}, we require $\Theta(\tilde{g}(n))$ to converge to zero at a sufficiently fast polynomial rate, and hence we do need to take a representation of
the form $\tilde{g}(n)=n^{\gamma}$. The numerical methods above may be adapted further to study
the largest such $\gamma$ we can take. Moreover
we observe that there is a blow up in the confidence intervals 
of the estimations due to slow polynomial decay of correlations (and hence no CLT convergence). 
This demonstrates that, despite the numerical procedure focuses on checking assumption (H2a), 
it is also sensitive to a deterioration on the decay of correlations. 
Hence, the procedure also provides an indirect check of assumption (H1). 

\paragraph{The Alves-Viana map}
As introduced in Section \ref{sec.partial}, consider the Alves-Viana map defined by 
equation \eqref{eq.av}. Numerically, we observe that $\nu(E_n(\gamma))=O(n^{-0.47})$ 
for $\gamma=0.5$, see Figure \ref{fig.av}. The bound obtained in 
Theorem \ref{thm.alvesviana} is taken for functions of the form $\tilde{g}(n)=(\log n)^{\gamma}$, and
hence we might expect a faster asymptotic estimate for the rate of convergence to an EVD.

\subsection{Decay of $\nu(E_n(\gamma))$ for non-uniformly hyperbolic systems}\label{sec.numerical.hyp} 

\begin{figure}[t]
\centering
\psfrag{n}{$n$}
\psfrag{meas}{$\nu(E_n(\gamma))$}
\includegraphics[width=0.49\textwidth]{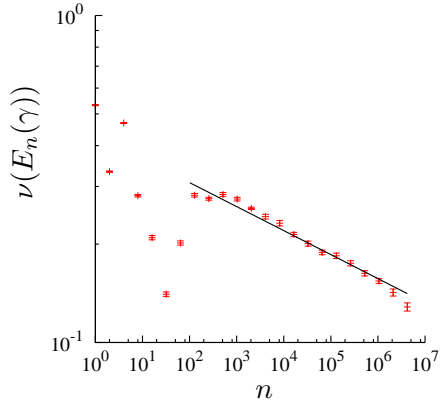}
\caption{
As Figure \ref{fig.doubling}, but for the H\'enon map \eqref{eq.henon}
with $(a,b)=(1.4,0.3)$ and $\gamma=0.5$. 
} 
\label{fig.henon}
\end{figure}

\begin{figure}[t]
\centering
\psfrag{n}{$n$}
\psfrag{meas}{$\nu(E_n(\gamma))$}
\includegraphics[width=0.49\textwidth]{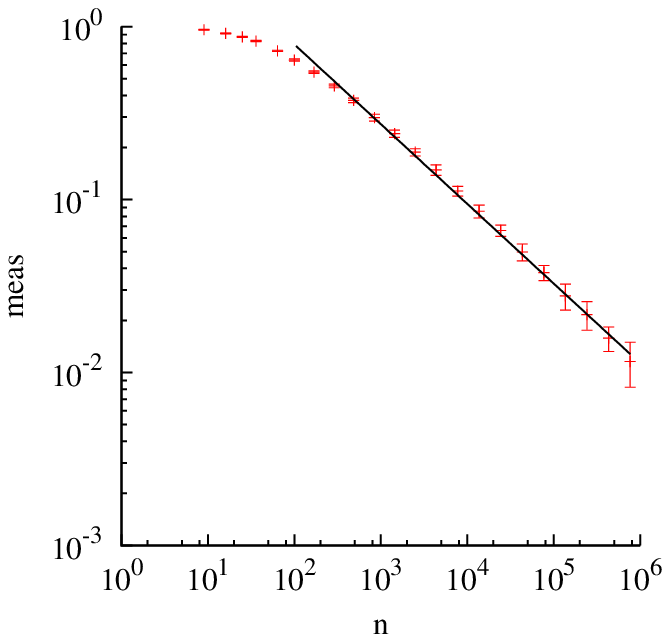}
\includegraphics[width=0.49\textwidth]{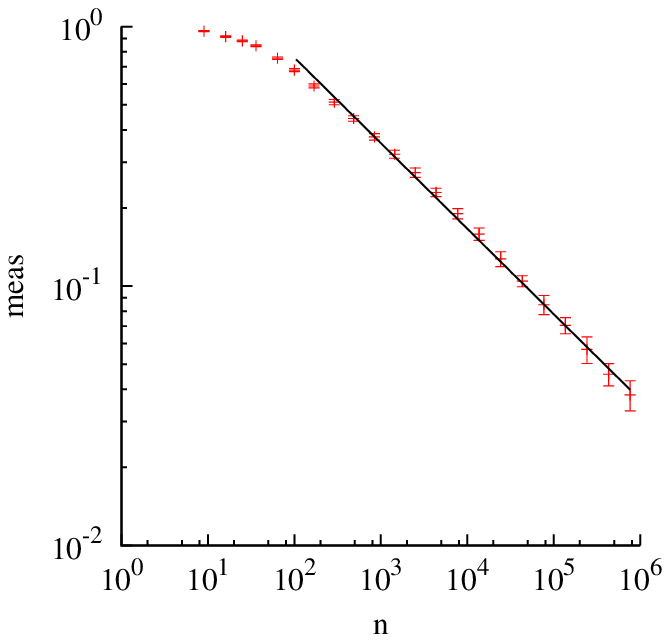}
\caption{
As Figure \ref{fig.doubling}, but for the Axiom-A diffeomorphism given by 
\eqref{eq.anos} for $\varepsilon =0$ (left) and $\varepsilon =0.15$ (right). 
In both cases $\gamma=0.5$. 
}
\label{fig.anosov}
\end{figure}

The numerical procedure of Section~\ref{sec.numerics} can
be applied also to systems that fit the hypothesis of Theorem~\ref{thm.extreme2}. 
As in the previous section, the numerical analysis of these two maps 
also shows a power decay law as expected, In this section, let us analyse maps 
which are not known to satisfy (H2a)-(H2c), such as the H\'enon map
(for the classic parameter values) and Axiom-A diffeomorphisms 
(with rank of the unstable dimension greater than one). 

\paragraph{The H\'enon map.}

The H\'enon family is given by:
\begin{equation}\label{eq.henon}
f(x,y)=(1-ax^2+y,\, bx),\quad (a,b)\in\mathbb{R}^2.
\end{equation}
By the theory of \cite{BY1} it is shown that the system admits a Young
tower with exponential decay of correlations. Following \cite{CC} and the discussion
of Section \ref{sec.SRT} the conditions of Theorem
\ref{thm.extreme} hold for a positive measure subset the parameter space. However this parameter set is not 
readily computable, and it is an open problem to determine whether 
there is a strange attractor for the parameters $(a,b)=(1.4,0.3)$. 
Numerical investigations at these parameters suggest that 
$\nu(E_n(\gamma))=O(n^{-0.073})$ for $\gamma=0.5$, and
thus power law behaviour is observed (albeit at a weak rate), see Figure \ref{fig.henon}. Hence we
expect convergence to an extreme value law to hold (in the sense described in Theorem
\ref{thm.extreme2}). As remarked for uniformly expanding maps, 
the convergence rate estimate can be improved further by taking the function 
$\tilde{g}(n)$ of logarithmic type.

\paragraph{Axiom-A diffeomorphisms}

For non-uniformly hyperbolic systems, analytic proofs on convergence to 
an EVD are generally achieved for rank 1 attractors, i.e., where 
the dimension of the unstable conditional measures is equal to $1$. 
However, for Axiom A systems little is known on return
time statistics and convergence to EVD when the unstable dimension 
is greater than $1$. 

Consider the map in the $3$-dimensional torus $f:\mathbb{T}^3 
\rightarrow \mathbb{T}^3$, with $\mathbb{T} = \mathbb{R}/ \mathbb{Z}$,  
defined as
\begin{equation} \label{eq.anos}
f(x_1, x_2, x_3) = (x_1 + x_2 + x_3 + \varepsilon \cos(2\pi x_2), \, 
x_1 + 2 x_2, \, x_1 + x_3 ). 
\end{equation}
For $\varepsilon$ small this map is an Anosov Map with a two dimensional unstable 
manifold. Our numerical computations indicate that $\nu(E_n(\gamma))=
O(n^{-0.46})$ for $\varepsilon=0$  and $\nu(E_n(\gamma))=O(n^{-0.33})$ 
for $\varepsilon=0.15$. For both cases we have taken $\gamma=0.5$. 

\subsection{Decay of $\nu(E_n(\gamma))$ for quasi-periodic systems} \label{sec.num.quasiperiodic}

\begin{figure}[t]
\centering
\psfrag{n}{$n$}
\psfrag{meas}{$\nu(E_n(\gamma))$}
\includegraphics[width=0.49\textwidth]{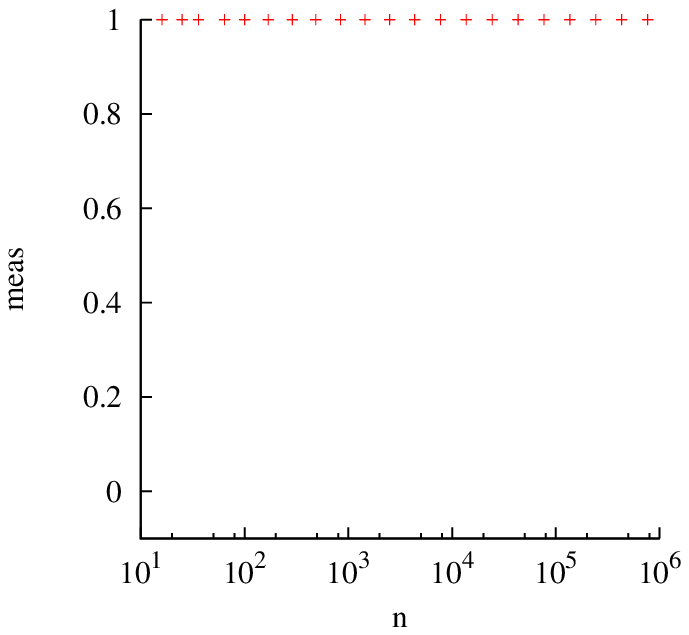}
\includegraphics[width=0.49\textwidth]{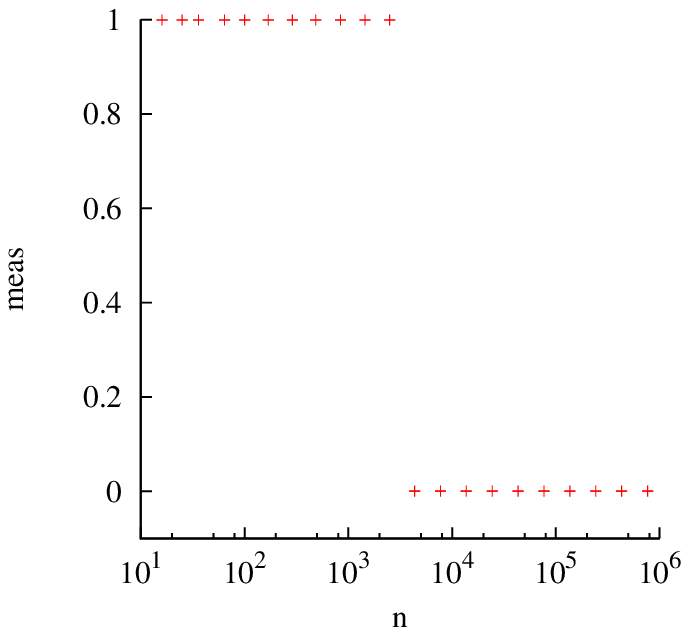}
\caption{
As Figure \ref{fig.doubling}, but for the Arnold family \eqref{eq.arnold} with
$\theta=1/2$ (left) and $\theta=1/3$ (right). In both cases
we have used $k=0.1$ and $\gamma=0.5$. 
Note that only the horizontal axis shows a logarithmic scale.}
\label{fig.arnold}
\end{figure}

For quasi-periodic systems, we observe that the recurrence statistics do not conform to 
conditions (H2a)-(H2c), at least 
for typical irrational rotation numbers. We verify the behaviour of 
$\nu(E_n(\gamma))$ for the Arnold Family $f:S^1\to S^1$ given by
\begin{equation}
\label{eq.arnold}
f(x)=x+\theta+k\sin(2\pi x),\quad\theta\in(0,1),\quad k<\frac{1}{2\pi}. 
\end{equation}
See Figure~\ref{fig.arnold}.
For $\theta=1/2$ the map exhibits phase locking and the attractor 
is a period two orbit. For $\theta=1/3$ the family displays quasi-periodic 
behaviour.  We observe that $\nu(E_n(\gamma))$ 
takes values either equal to zero or to one as $n$ varies, as 
predicted in Section~\ref{sec.quasiperiodic}.

\subsection{The Lorenz equations}\label{sec.lorenz}

\begin{figure}[t]
\centering
\psfrag{n}{$n$}
\psfrag{meas}{$\nu(E_n(\gamma))$}
\includegraphics[width=0.49\textwidth]{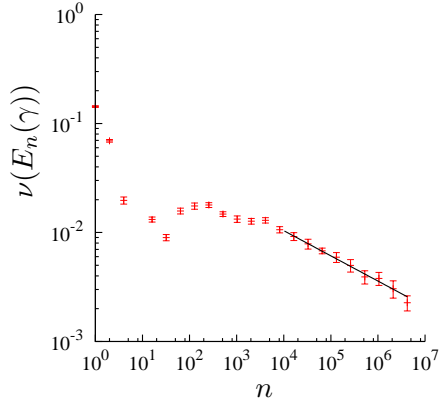}
\caption{
As Figure \ref{fig.doubling}, but for Poincar\'e return map of the Lorenz--63
flow with the section $\Sigma=\{x_3=30\}$. We have used with $\gamma=0.3$;
the initial condition was obtained by iterating the map 500 times using
$(0,0)$ as a starting point. The estimates of $\nu(E_n)$ exhibit
a long transient behaviour, and only the estimates for $n\geq 10^4$ are
fitted to a straight line.}
\label{fig.lorenz63map}
\end{figure}

\begin{figure}[t]
\centering
\psfrag{T}{$T$}
\psfrag{meas}{$\nu(E_T(\gamma))$}
\includegraphics[width=0.49\textwidth]{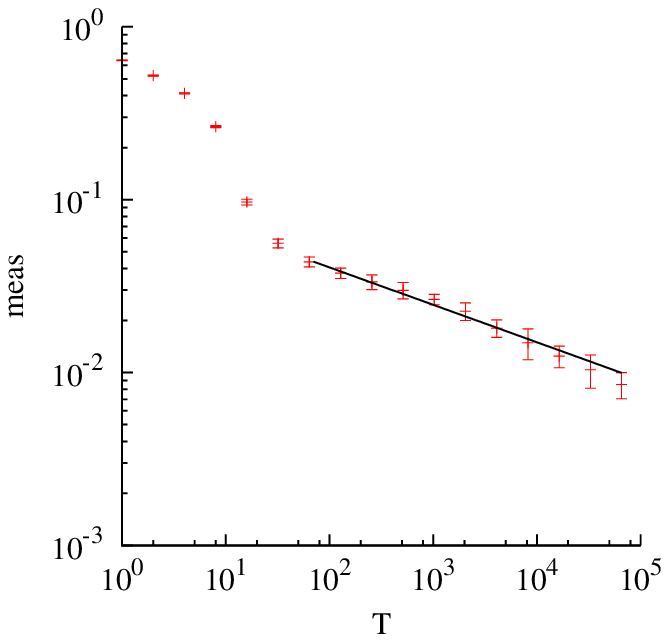}
\includegraphics[width=0.49\textwidth]{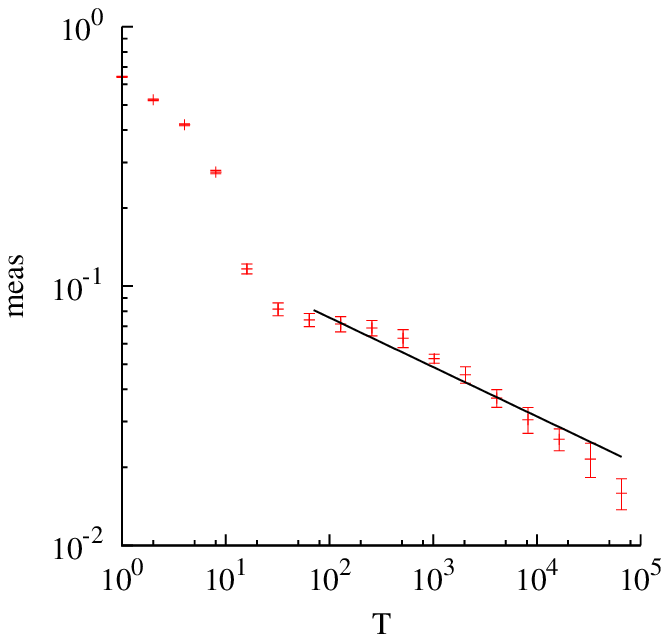}
\caption{
As Figure \ref{fig.doubling}, but for stroboscopic map of the Lorenz--63
flow with the sampling time $\tau=0.01$ (left) and $\tau=0.001$ (right).
In both cases $\gamma=0.4$. 
}
\label{fig.lorenz63strobo}
\end{figure}

Consider the Lorenz--63 equations
\[
\dot{x}_1 = \sigma(x_2-x_1), \quad
\dot{x}_2 = x_1(\rho-x_3) - x_2, \quad
\dot{x}_3 = x_1x_2-\beta x_3,
\]
with the classical parameters $(\sigma,\rho,\beta)=(10,28,8/3)$. From the
resulting semi-flow we can derive two maps: a 2-dimensional Poincar\'e return map
\begin{equation}\label{eq.returnmap}
f_{\rm return} : \Sigma \to \Sigma
\end{equation}
defined on a section $\Sigma$ transverse to the flow, and a 3-dimensional
stroboscopic map
\begin{equation}\label{eq.strobomap}
f_{\rm strobo} : \mathbb{R}^3 \to \mathbb{R}^3
\end{equation}
obtained by sampling the flow at multiples of a chosen sampling time $\tau$.

For the return map \eqref{eq.returnmap} with the section $\Sigma=\{x_3=30\}$ we
observe for $\gamma=0.3$ that $\nu(E_n(\gamma))=O(n^{-0.23})$, see Figure
\ref{fig.lorenz63map}. Let $F_t$ denote the semi-flow of the Lorenz--63 equations.
Then the measure of
\[
E_T(\gamma)
=
\left\{
x\in\mathbb{R}^3:\, \dist(x,F_t(x)) \leq\frac{1}{T^{1/3}} \;\textrm{for some $\delta_0 \leq t \leq T^{\gamma}$}
\right\}
\]
can be approximated by the measure of
\[
E_T(\gamma)
=
\left\{
x\in\mathbb{R}^3:\, \dist(x,f_{\rm strobo}^j(x)) \leq\frac{1}{T^{1/3}} \textrm{ for some } \frac{\delta_0}{\tau} \leq j \leq \frac{T^{\gamma}}{\tau}
\right\}.
\]
Fix $\delta_0=0.01$ and $\gamma=0.4$. For $\tau=0.01$ we observe $\nu(E_T(\gamma)) = O(T^{-0.21})$,
and for $\tau=0.001$ we observe $\nu(E_T(\gamma)) = O(T^{-0.19})$, 
see Figure \ref{fig.lorenz63strobo}. Hence we can conjecture that the rate of convergence
to an EVD for the Lorenz equations is of the order $O(T^{-\alpha})$ for some $\alpha>0.2$.
This appears to an improvement on the sub-polynomial bounds obtained in Theorems \ref{thm.lorenz1}
and \ref{thm.lorenz2} for the corresponding discrete time Poincar\'e maps. We note that proof
of convergence to an EVD for the Lorenz flow (without an error rate) was recently established in \cite{Licheng}.


\section{Proofs of the main results}\label{sec.proofs}
In this section we begin by proving Propositions \ref{prop.extreme1} and \ref{prop.extreme2}.
Finally we prove Theorem \ref{thm.stretched}.

\subsection{Proof of Proposition \ref{prop.extreme1}}
We begin by proving Case 3 in the case where $(f,\XX,\nu)$ is a non-uniformly hyperbolic system. 
The argument simplifies in the non-uniformly expanding case. In particular Case 1 
is proved directly in \cite{holland-nicol}. We will show how Case 2 follows.

For systems with exponential decay of correlations an estimate for $\gamma(n,t)$ is given in \cite{GHN}.
We give the appropriate modifications here in the case of subexponential decay of correlations. The proof relies on the Young tower construction for non-uniformly hyperbolic systems as detailed in \cite{Young1}. The notations we present
here are consistent with those of \cite{GHN}, and in particular we refer to this reference when
there is a strong overlap in the argument.

Consider the set
\[
B_{r,k}(\tilde{x}):= \left\{x:f^{k}(\gamma^s(x))\cap \partial B (\tilde{x}, r)\ne\emptyset\right\},
\]
where $\gamma^s(x)$ is the local stable manifold through $x$, (which exists $\nu$-a.e.). 
We begin with the following estimate:
\begin{lemma}\label{lem:annulus1}
Under assumption (H3) there exist constants $C>0$ and $\tau_1<1$ such that for any $r, k$:
 \[
  \nu(B_{r, k}(\tilde{x}))\le  C \tau_1^{k}.
 \]
\end{lemma}

\begin{proof}
From the construction of the Young tower, in particular \cite[Property P3]{Young1} 
there exists an $\tau\in(0, 1)$ and a $C>0$ such that $\dist(f^n(x), f^n(y))\le C\tau^n$ for all 
$y\in \gamma^s(x).$ In particular, this implies that $|f^k( \gamma^s(x) )|\le C\tau^k$ where $|\dots|$ denotes the length with respect to the Lebesgue measure. Therefore, $f^k(B_{r,k}(\tilde{x}))$ lies in an annulus of width $2C\tau^k$ around the boundary of the ball of radius $r$ centered at the point $\tilde{x}$. By Assumption (H3) and invariance of $\nu$ the existence of such a $\tau_1$ 
follows. The constant $\tau_1$ depends on both $\sigma$ and $\tau$.
\end{proof}
We now estimate of $\gamma(n,t)$ using estimate on decay of correlations.
\begin{lemma}\label{lemma:dun-prelim}
Suppose
$\Phi : \XX \to \mathbb{R}$ is Lipschitz and $\Psi_{j,l}$ is the indicator function 
\[
\Psi_{j,l} := 1_{\left\{X_{j+1}\le u_n, X_{j+1}\le u_n, \dots, X_{j+l}\le u_n\right\}}.
\]
Let $\alpha_j$ be an increasing sequence with $\alpha_j\leq j$. Then
\begin{equation}
\left|\int\Phi\Psi_{0,l}\circ f^j d\nu - \int\Phi\text{d}\nu\int\Psi_{0,l}d\nu\right|
=
O(1)\left(\|\Phi\|_\infty \tau_1^{\alpha_j}+\|\Phi\|_{\text{Lip}}\Theta(j-\alpha_j)\right).
\end{equation}
\end{lemma}

\begin{proof}
The proof follows \cite[Lemma 3.1]{GHN} where in our case we just keep track
of the decay of correlation term $\Theta(j)$ (which need not be exponentially 
fast in our case). We present the main details, avoiding as far as possible the technical
construction of the Young tower.
Define the functions $\tilde{\Phi} : \XX\times\mathbb{N} \to \mathbb{R}$ and
$\tilde{\Psi} : \XX\times\mathbb{N} \to \mathbb{R}$ by
\[
\tilde{\Phi}(x,r) = \Phi(f^r(x))
\quad\mbox{and}\quad
\tilde{\Psi}_{j,l}(x,r) = \Psi_{j,l}(f^r(x)).
\]
In the construction of the Young tower, there exists a
a reference set $\Lambda\subset\XX$ with a hyperbolic product structure,
and we can choose a reference unstable manifold 
$\hat\gamma^u\subset \Lambda$ with the property that
for any $x\in\Lambda$, there exists $\hat{x}\in\hat\gamma^u$ with
$\gamma^s(x)\cap\hat\gamma^u=\{\hat{x}\}$.

Define the function $\overline\Psi_{j,l}(x,r) := \tilde{\Psi}_{j,l}(\hat x, r)$. The function $\overline \Psi_{j,l}$ is constant along stable manifolds in $\Lambda$.
The set $\{\overline \Psi_{j,l}\neq\tilde\Psi_{j,l}\}$ consists of points $(x,r)\in\XX\times\mathbb{N}$, with the property
that there exist $x_1,x_2\in\gamma^s(f^r(x))$ such that 
\[
x_1\in \{X_j  \le  u_n, \dots, X_{j+l}\le u_n\},\;\textrm{but}\;
x_2 \notin \{X_j  \le  u_n, \dots, X_{j+l}\le   u_n \}.
\]
This set is contained inside $\cup_{k=j}^{j+l}f^{-k}(B_{u_n,k})$.
By Lemma \ref{lem:annulus1} we  have
\[
\nu\set{\tilde\Psi_{\alpha_j,l}\neq\overline\Psi_{\alpha_j,l}}
\leq
\sum_{k=\alpha_j}^{l}\nu(B_{u_n, k})
=
O(\tau_1^{\alpha_j}).
\] 
In \cite[Lemma 3.1]{GHN} it is shown that for $\alpha_j<j$:
\begin{equation}
\abs{\int \Phi\Psi_{0,l}\circ f^{j}d\nu - \int\Phi d\nu\int \Psi_{0,l}d\nu}=
\abs{\int \Phi\Psi_{\alpha_j,l}\circ f^{j -\alpha_j}d\nu - \int\Phi d\nu\int \Psi_{\alpha_j,l}\circ f^{j-\alpha_j}d\nu}.
\end{equation}
To complete the proof we note that the second term on the right is bounded by:
\begin{gather*}
\abs{\int \Phi(\Psi_{\alpha_j,l}-\overline\Psi_{\alpha_j,l})\circ f^{j -\alpha_j}d\nu - \int\Phi d\nu\int 
(\Psi_{\alpha_j,l}-\Psi_{\alpha_j,l})\circ f^{j-\alpha_j}d\nu}\\
+
\abs{\int \Phi\overline\Psi_{\alpha_j,l}\circ f^{j -\alpha_j}d\nu - \int\Phi d\nu\int \overline\Psi_{\alpha_j,l}d\nu}\\
\leq 
O(1)\left(\|\Phi\|_\infty\nu\set{\overline\Psi_{\alpha_j,l}\neq\tilde\Psi_{\alpha_j,l}}+\|\Phi\|_{\text{Lip}}
\Theta(j-\alpha_j)\right)\\
\le O(1)\left(\|\Phi\|_\infty  \tau_1^{\alpha_j}+\|\Phi\|_{\text{Lip}}\Theta(j-\alpha_j)\right).
\end{gather*}
\end{proof}
To apply Lemma \ref{lemma:dun-prelim} in order to bound $\gamma(n,t)$ we must first 
approximate $\Phi := 1_{\{X_1 > u_n\} }$ by a Lipschitz continuous function $\Phi_B$.
We do this as follows. The set $\{X_1 > u_n \}$ corresponds to
a ball of radius $\ell$ centered at the point $\tilde{x}$. We define $\Phi_B$ to be 1
inside a ball centered at $\tilde{x}$ of radius $\ell$ and decaying to 0 at a linear
rate on the annulus $A(\tilde{x}, n):=B(\tilde{x},\ell')\setminus
B(\tilde{x},\ell)$ so that on the boundary of $B(\tilde{x},\ell')$, $\Phi_B$ vanishes.
The Lipschitz norm of $\Phi_B$ is seen to be bounded by $1/(\ell'-\ell)$.
We have the following estimate:
\begin{equation}\label{eq.correst1}
\begin{split}
 \abs{\int \Phi \Psi_{0,l}\circ f^{j}d\nu - \int\Phi\,d\nu\int \Psi_{0,l}d\nu}
&\leq \abs{\int \Phi_B \Psi_{0,l}\circ f^{j}d\,\nu - \int\Phi_B\,d\nu\int \Psi_{0,l}d\nu}\\
&+\abs{\int (\Phi-\Phi_B) \Psi_{0,l}\circ f^{j}d\,\nu - \int(\Phi-\Phi_B)\,d\nu\int \Psi_{0,l}d\nu},\\
\end{split}
\end{equation}
If we set $\lambda=(\ell'-\ell)$, then for some $\sigma_1>0$ (coming from (H3)) we see that equation
\eqref{eq.correst1} is bounded by
\begin{gather*}
C\left(\|\Phi\|_\infty \tau_1^{\alpha_j}+\lambda^{-1}\Theta(j-\alpha_j)\right)
+(1+\|\Psi_\infty\|)\|\Phi-\Phi_B\|_1\\
\leq C\left(\|\Phi\|_\infty \tau_1^{\alpha_j}+\lambda^{-1}\Theta(j-\alpha_j)\right)
+(1+\|\Psi_\infty\|)\lambda^{\sigma_1}.\\
\end{gather*}
If we put $\alpha_j=j/2$ then we have:
\[
\gamma(n,j)\leq O(1)\left(\tau^{j/2}_1+\lambda^{\sigma_1}+ 
\lambda^{-1}\Theta(j/2)\right).
\]
Regarding the right hand side as a real valued function of $\lambda$, a simple calculus argument implies
that $\lambda=O(1)\Theta(j/2)^{\frac{1}{\sigma_1+1}}$ is the minimizer. Hence we deduce that there exist constants
$\tilde\tau>0$ and $\tilde\sigma>0$ such that
\[
\gamma(n,j)\leq O(1)\max\{\tilde\tau^{j},\Theta(j/2)^{\tilde\sigma}\}.
\]
This completes the proof for Case 3.

For Case 2, we do not need (H3), and the Lipschitz approximation argument above leads directly to the estimate:
\begin{equation}
\begin{split}
\gamma(n,t) &\leq O(1)\left(\|\Phi_B\|_{\mathrm{Lip}}\Theta(t)+\nu\{\Phi_B\neq\Phi\} \right),\\
&\leq O(1)\left( \lambda^{-1}\Theta(t)+\exp\{-C_d|\log\lambda|^{\delta}\}\right),
\end{split}
\end{equation}
where the constant $C_d$ depends on the dimension $d$, and $\delta$ is the constant in the $L^p$ norm
of the density $\rho$. Again, if we consider the right hand side as a function
of $\lambda$, a calculus argument leads to an approximate minimizer:
$$\hat\lambda=\Theta(t)\exp\{C_d|\log\Theta(t)|^{\delta}\}.$$ For this value of $\hat\lambda$, we get the bound:
\begin{equation}
\gamma(n,t)\leq O(1)\exp\left\{-|\log\Theta(t)|^{\delta_1}\right\}
\end{equation}
where $\delta_1$ is a uniform constant (depending on $\delta$). This completes the proof of Case 2.


\subsection{Proof of Proposition \ref{prop.extreme2}}
We start by proving Proposition \ref{prop.extreme2} for non-uniformly expanding systems under the assumptions of (H2b) and (H2c). For 
non-uniformly expanding systems satisfying (H2a) a version of this proposition 
was proved in \cite{holland-nicol} (using the same approach). For 
non-uniformly expanding systems the ergodic measure $\nu$ is absolutely 
continuous 
with respect to Lebesgue (volume) measure. For systems that have less regular
SRB measures (such as those singular with respect to volume), we show how the 
relevant arguments are adapted.

\paragraph{Proof in the non-uniformly expanding case}
For a function $\varphi\in L^1(m)$ we define the Hardy--Littlewood maximal 
function
\[
\mathcal{M}(x):=\sup_{a>0}\frac{1}{m(B(x,a))}\int_{B(x,a)}\varphi(y) dm(y).
\]
A theorem of Hardy and Littlewood~\cite{Rudin},  implies that
\begin{equation}\label{eq.HLbound}
m(|\mathcal{M}(x)|>\lambda)\le \frac{\|\varphi \|_1}{\lambda},
\end{equation}
where $\|\cdot\|_1$ is the $L^1$ norm with respect to $m$.
Recalling
\[
E_n
=
\left\{x:\,\dist(x,f^j(x))\leq\frac{1}{n^{1/d}}\,\textrm{for some}\,j\leq\tilde{g}(n)\right\},
\]
let $\rho(x)$ denote the density of $\nu$ with respect to $m$ and let $\mathcal{M}_n(x)$ denote the
maximal function of $\varphi_{n}(x):=1_{E_n}(x)\rho(x)$.
For constants $a,b>0$ to be fixed later consider sequences $\alpha_n=e^{n^{b}}$
and $\lambda_n=n^{-a}$ for some $b\in(0,1)$ and $a>0$. Inequality
(\ref{eq.HLbound}) gives
\[
m(|\mathcal{M}_{\alpha_n}(x)|>\lambda_n)
\leq
\frac{\nu(E_{\alpha_n})}{\lambda_n}
\leq
C(\log\alpha_n)^{-\alpha}n^{a}=Cn^{a-b\alpha}.
\]
If $\alpha b-a>1$ (first constraint required on $a$ and $b$), then
the First Borel--Cantelli Lemma implies for $\nu$ a.e. $x$ there exists an $N:=N(x)$
such that for all $n\geq N$ we have $|\mathcal{M}_{\alpha_n}(\tilde{x})|<\lambda_n$.
For a non-uniformly expanding system recall that $\nu$ is assumed 
absolutely continuous with respect to $m$. 
Hence, for all $n$ sufficiently large
\begin{equation}\label{eq.meas-ball}
\begin{split}
\nu(\left\{x:\dist(x,\tilde{x})<\alpha_n^{-1/d}\right\}\cap E_{\alpha_n})
& \leq
\int_{B(\tilde{x},\alpha^{-1}_n)}\varphi_{\alpha_n}(y) dm(y) \\
& \leq
C_d\alpha^{-1}_n\mathcal{M}_{\alpha_n}(\tilde{x}) \\
&\leq C_d\alpha^{-1}_n\lambda_n=O\left( e^{-n^{b}}n^{-a}\right),
\end{split}
\end{equation}
where $C_d$ depends on the dimension.
Denote $A := \{X_1>u_k, X_j>u_k\}$ with $2 \leq j \leq g(k)$, and $g(n)=\tilde{g}(n)^{(1-\epsilon)}$ for
some $\epsilon>0$. We assume that $\tilde{g}(n)$ has the representations given in (H2b). For observables of the form $\phi(x)=\psi(\dist(x,\tilde{x}))$, we have $\psi^{-1}(u_k)\approx 1/k^{1/d}$. Hence there exists a $v>0$ such that
\[
A\subset
\left\{x:\dist(\tilde{x},x)\leq\frac{v}{k^{1/d}} \; \dist(\tilde{x},f^j(x))\leq
\frac{v}{k^{1/d}} \textrm{ for some } j\leq g(k)\right\}. 
\]
Given the sequence $\alpha_n$, let $k^{1/d}/(2v)\in[\alpha_n,\alpha_{n+1})$. Then (by monotonicity of $g(n)$),
\[
A \subset
\left\{x:\dist(\tilde{x},x)\leq\frac{1}{2\alpha_n}, \; \dist(\tilde{x},f^j(x))\leq\frac{1}{2\alpha_n} \textrm{ for some } j\leq g((2v)\alpha_{n+1})\right\}.
\]
Applying the triangle inequality $\dist(x,f^j(x)) \leq \dist(\tilde{x},x) + \dist(\tilde{x},f^j(x))$ gives
\[
A
\subset
\left\{x:\dist(\tilde{x},x)\leq\frac{1}{\alpha_n}, \; \dist(x,f^j(x))\leq\frac{1}{\alpha_n} \textrm{ for some } j\leq g((2v)\alpha_{n+1})\right\}.
\]
Since $\alpha_n=\lfloor e^{n^{b}}\rfloor $ (with $b\in(0,1)$) we have that 
$\lim_{n\to\infty}\alpha_{n+1}/\alpha_n=1$. By
the growth properties of $g$ and $\tilde{g}$ (as given in Proposition 
\ref{prop.extreme2}), there exists $\kappa_{v}>0$ and a sequence $c_n\to 0$ such that for all sufficiently large $\alpha_n$: $$g((2v)\alpha_{n+1})\leq g(2(2v)\alpha_n)\leq c_n\tilde{g}((2(2v)\alpha_n)\leq c_n\kappa_v\tilde{g}(\alpha_n).$$ 
Moreover, there exists $N$ such that $\forall n\geq N$ we have $c_n\kappa_v<1$, and hence
\[
A
\subset
\left\{x:\dist(\tilde{x},x)\leq\frac{1}{\alpha_n}, \; \dist(x,f^j(x))\leq\frac{1}{\alpha_n} \textrm{ for some } j\leq \tilde{g}(\alpha_n)\right\}.
\]
Applying inequality (\ref{eq.meas-ball}) we obtain:
\[
\nu(X_1>u_k, X_j>u_k) =  O\left(k^{-1}(\log k)^{-a/b}\right) \quad\mbox{for all } k > N,
\]
so that
\begin{equation}\label{eq.maximal.final}
\sum_{j=1}^{g(k)} \nu (X_1 > u_k, X_j > u_k)
=
O\left(\frac{g(k)}{k(\log k)^{a/b}}\right) \quad\mbox{for all } k > N.
\end{equation}
To complete the proof, we now optimize the ratio $a/b$ appearing in equation \eqref{eq.maximal.final},
subject to $a>0$, $b\in(0,1)$ and  $\alpha b-a>1$. We find that the optimal ratio 
$\tilde\alpha$ can be chosen arbitrarily close to $\alpha-1$. Hence we
obtain for all $k>N$:
\[
\sum_{j=1}^{g(k)} \nu (X_1 > u_k, X_j > u_k)
=O(k^{-1}(\log k)^{-\tilde\alpha}g(k)),
\]
valid for all $\tilde\alpha<\alpha-1$. Hence equation \eqref{eq:dprime2} is satisfied.

We now consider the case of (H2c). Using the maximal function argument above, 
we take 
sequences $\alpha_n=e^{n^{b}}$ and $\lambda_n=e^{-n^{a}}$ for some $a,b\in(0,1)$. 
In this case for $\nu$-a.e. $\tilde{x}$, there exists $N$ such that
\begin{equation}
\sum_{j=2}^{g(k)} \nu (X_1 > u_k, X_j > u_k)
\leq O(1) g(k)k^{-1}\exp\{-(\log k)^{a/b}\} ,\quad\mbox{for all } k > N.
\end{equation}
Again we can optimize the ratio $a/b$ subject to the constraint $\alpha>a/b$. The optimal
ratio $\tilde\alpha$  can be chosen arbitrarily close to $\alpha$.

\paragraph{Proof in the non-uniformly hyperbolic case}
Consider the case under the general assumption that $\nu$ is
SRB with local dimension $d_{\nu}$. We work with the following maximal 
function $\mathcal{M}_{\nu}$ where
\[
\mathcal{M}_{\nu}(x):=\sup_{a>0}\frac{1}{\nu(B(x,a))}\int_{B(x,a)}\varphi(y) 
d\nu(y).
\]
An application of the Besicovitch covering theorem \cite{fefferman} gives
\begin{equation}\label{eq.Bes-bound}
\nu(|\mathcal{M}_{\nu}(x)|>\lambda)\le \frac{\|\varphi \|_1}{\lambda},
\end{equation}
where in this case $\|\cdot\|_1$ is the $L^1$ norm with respect to $\nu$.
It suffices to show how the calculations in the non-uniformly
expanding case extend. We will do this under the assumption of the
weak recurrence condition (H2b). The same argument can be applied under
assumptions (H2a) or (H2b).

As in the non-uniformly expanding case, for constants $a,b>0$ let 
$\alpha_n=e^{n^{b}}$ and $\lambda_n=n^{-a}$ for some $b\in(0,1)$ and $a>0$. 
As before, if $\alpha b-a>1$, inequality \eqref{eq.Bes-bound} implies that for $\nu$ a.e. $x$ there exists an 
$N:=N(x)$ such that for all $n\geq N$ we have 
$|\mathcal{M}_{\alpha_n}(\tilde{x})|<\lambda_n$.
Hence, for all $n$ sufficiently large
\begin{equation}\label{eq.meas-ball2}
\begin{split}
\nu(\left\{x:\dist(x,\tilde{x})<\alpha_n^{-1/d}\right\}\cap E_{\alpha_n})
& \leq
\int_{B(\tilde{x},\alpha^{-1/d}_n)}\varphi_{\alpha_n}(y) d\nu(y) \\
& \leq
\nu(B(\tilde{x},\alpha^{-1/d}_n))\mathcal{M}_{\alpha_n}(\tilde{x}) \\
&\leq \nu(B(\tilde{x},\alpha^{-1/d}_n))\lambda_n.
\end{split}
\end{equation}
Denote $A := \{X_1>u_k, X_j>u_k\}$ with $2 \leq j \leq g(k)$, and $g(n)=\tilde{g}(n)^{(1-\epsilon)}$ for
some $\epsilon>0$. We assume that $\tilde{g}(n)$ has the representations given in (H2b). For observables of the form 
$\phi(x)=\psi(\dist(x,\tilde{x}))$, recall that the sequence $u_n$ is chosen
so that $n\nu\{\dist(x,\tilde{x})\leq \psi^{-1}(u_n)\}\to\tau$. This implies
that we have $\psi^{-1}(u_k)\in[1/k^{d_{\nu}+\epsilon},1/k^{d_{\nu}-\epsilon}]$, where
the constant $\epsilon$ is due to the fluctuation of $\nu$ on small balls.
If we let $w_k =\psi^{-1}(u_k)$, then
\[
A\subset
\left\{x:\dist(\tilde{x},x)\leq w_k\; \dist(\tilde{x},f^j(x))\leq w_k \textrm{ for some } j\leq g(k)\right\}. 
\]
Given the sequence $\alpha_n$, let $(1/2w_k)^{d}\in[\alpha_n,\alpha_{n+1})$. 
Using
again the monotonicity of $g(n)$ and the triangle inequality we obtain
\[
A
\subset
\left\{x:\dist(\tilde{x},x)\leq\frac{1}{\alpha^{1/d}_n}, \; \dist(x,f^j(x))\leq\frac{1}{\alpha^{1/d}_n} \textrm{ for some } j\leq g(\alpha^{\tilde{d}}_{n+1})
\right\},
\]
where $\tilde{d}>0$ is a constant depending on $d$ and $d_{\nu}$.
Since $\alpha_n=\lfloor e^{n^{b}}\rfloor $ (with $b\in(0,1)$) we have that 
$\lim_{n\to\infty}\alpha_{n+1}/\alpha_n=1$. By
the growth properties of $g(n)$ and $\tilde{g}(n)$ (as given in Proposition 
\ref{prop.extreme2}), there exists a sequence $c_n\to 0$ such that for all 
sufficiently large $\alpha_n$: $$g(\alpha^{\tilde{d}}_{n+1})
\leq c_n\tilde{g}(\alpha_n).$$ 
Moreover, there exists $N$ such that $\forall n\geq N$ we have $c_n<1$, 
and hence
\[
A
\subset
\left\{x:\dist(\tilde{x},x)\leq\frac{1}{\alpha^{1/d}_n}, \; \dist(x,f^j(x))\leq\frac{1}{\alpha^{1/d}_n} \textrm{ for some } j\leq \tilde{g}(\alpha_n)\right\}.
\]
We now use the precise asymptotics of $\alpha_n$ together with the fact that
$k$ is chosen so that $(1/2w_k)^{d}\in[\alpha_n,\alpha_{n+1})$. Due
to fluctuations of $\nu$ on small balls, we cannot achieve the same 
inequality obtained in \eqref{eq.meas-ball}. In general 
the measure $\nu(B(\tilde{x},\alpha^{-1/d}_n))$ is not uniformly
comparable to  $\nu(B(\tilde{x},\alpha^{-1/d}_{n+1})$. Noting
that $2w_k\in[\alpha^{-1/d}_{n+1},\alpha^{-1/d}_{n}]$ and the fact
that $\alpha_n/\alpha_{n+1}\to 1$ we obtain: 
\begin{equation}\label{eq.meas-ball3}
\begin{split}
\nu(X_1>u_k, X_j>u_k) &\leq \nu(B(\tilde{x},4w_k))
\left(\log(1/4w_k)\right)^{-a/b}\\
&\leq C\nu(B(\tilde{x},4w_k))\left|\log\nu(B(\tilde{x},w_k))\right|^{-a/b},
\end{split}
\end{equation}
where $C$ depends on $d$ and $d_{\nu}$. By \eqref{eq.meas-ball2}, this
latter inequality holds for all $k$ sufficiently large (and $j<g(k)$).
To express the right-hand side of equation \eqref{eq.meas-ball3} in terms of $\nu(B(\tilde{x},w_k))$,
we use an iterated version of Lemma \ref{lem.Bes} given in section \ref{sec.SRT} as applied
to the set: 
$$\widetilde{\mathcal{F}}(\lambda,r,4):=\{x\in\XX:\nu(B(x,4r))\geq\lambda\nu(B(x,r))\}.$$
In particular we have
$$\widetilde{\mathcal{F}}(\lambda,r,4)
\subset\mathcal{F}(\sqrt{\lambda},2r)\cup  \mathcal{F}(\sqrt{\lambda},r),$$
and hence we have $\nu(\widetilde{\mathcal{F}}(\lambda,r))<C\lambda^{-1/2}.$
Along the sequence $w_k$, we have that for 
all $\tilde{x}\not\in\tilde{\mathcal{F}}(\lambda,w_k)$:
\begin{equation}
\nu(X_1>u_k, X_j>u_k)\leq C\lambda\nu(B(\tilde{x},w_k))\left|\log\nu(B(\tilde{x},w_k)\right|^{-a/b}.
\end{equation}
From Lemma \ref{lem.Bes} we specify $\lambda:=\lambda(r)=|\log r|^{2s}$, 
and take a subsequence $w_{k_n}=\beta_n:=e^{-n^c}$ for some $c\in(0,1)$. 
We choose $s>1$ chosen so that $cs>1$. Along the subsequence $w_{k_n}$ we have
that $\nu(\tilde{\mathcal{F}}(\lambda(\beta_n),\beta_n))\leq n^{-cs}$, and hence by the First Borel-Cantelli Lemma, there exists 
$n_0(\tilde{x})$ such that $\tilde{x}\not\in\widetilde{\mathcal{F}}(\lambda(\beta_n),\beta_n)$ 
for all $n>n_0$. In particular we have
\begin{equation}
\nu(X_1>u_{k_n}, X_j>u_{k_n})\leq C\nu(B(\tilde{x},\beta_n))\left|\log\nu(B(\tilde{x},\beta_n)\right|^{2s-a/b}.
\end{equation}
We now extend this estimate to all times $k$ such that $w_{k}<\beta_{n_0}$. 
Take $n>n_0$ and let 
$w_{k}\in[\beta_{n+1},\beta_{n}]$. Then we have:
\begin{equation}
\begin{split}
\nu(X_1>u_{k}, X_j>u_{k}) &\leq C\nu(B(\tilde{x},\beta_n))\left|\log\nu(B(\tilde{x},\beta_n)\right|^{2s-a/b}\\
&\leq C\lambda(\beta_n/4)\nu(B(\tilde{x},\beta_n/4))\left|\log\nu(B(\tilde{x},\beta_n/4)
\right|^{2s-a/b}\\
&\leq C\nu(B(\tilde{x},w_k))\left|\log\nu(B(\tilde{x},w_k)\right|^{4s-a/b}.
\end{split}
\end{equation}
In the second line we used Lemma \ref{lem.Bes}, and in the third line the fact that $w_{k}>\beta_{n+1}>\beta_{n}/4$.
It follows that
\[
\sum_{j=1}^{g(k)} \nu (X_1 > u_k, X_j > u_k)
=
O\left(\frac{g(k)}{k(\log k)^{-4s+a/b}}\right) \quad\mbox{for all } k > k_0.
\]
We now maximize $a/b-4s$ subject to $c\in(0,1)$, $sc>1$, and $\alpha b-a>1$. The optimal value
$\tilde\alpha$ can be chosen arbitrarily close to $\alpha-5$. 
\begin{rmk}\label{rmk.h2a.hyp} 
In the case of assumption (H2a), the proof is not so delicate. 
By the use of local dimension arguments Lemma \ref{lem.Bes} can be avoided.
Under assumption (H2a) we can take $\alpha_n=n^{-b}$ and $\lambda_n=n^{-a}$
for any $a,b>0$. If $b\alpha-a>1$ then for 
$\nu$-a.e $x\in\XX$, and for all $\epsilon>0$,
we have for all $k>k_0$:
\begin{equation}
\nu(X_1>u_k, X_j>u_k)\leq C\nu(B(\tilde{x},w_k))^{1-\epsilon-a/b}.
\end{equation}
The constant $\epsilon$ comes from the definition of local dimension, and
the constant $k_0$ depends on $\tilde{x}$ and $\epsilon$. If
we optimize over $a$ and $b$ it follows that
\[
\sum_{j=1}^{g(k)} \nu (X_1 > u_k, X_j > u_k)
=
O\left(\frac{g(k)}{k^{1-\alpha-\epsilon}}\right) \quad\mbox{for all } k > k_0.
\]
\end{rmk}

\subsection{Proof of Theorem \ref{thm.stretched}}\label{sec.proofstretched}
We prove Theorem \ref{thm.stretched} in several steps following the strategy presented in \cite{Collet}. 
In fact our proof
optimizes some of the calculations presented within, and allows us to deduce convergence to an 
EVD for non-uniformly expanding systems that either have an invariant density that does not belong to $L^p$, (for any $p>1$), and/or have sub-exponential decay of correlations. Along the way we get an estimate on the convergence rate.

In the first step we show that stretched exponential decay
of correlations implies that the invariant density $\rho(x)$ is bounded by function $h(x)$, with
$\int h(x)(\log h(x))^{q} dm<\infty$ for some $q>1$. We remark that $h(x)$ need not be in $L^{p}$ for any $p>1$. 
The result we present, namely Lemma \ref{lem.density} will also lead us to deduce a bound on the regularity of the invariant density for the Alves-Viana map. In the second step we show that intermediate quantitative recurrence statistics hold in the sense on (H2c). We then apply the blocking argument to deduce the convergence result with appropriate error term
in the convergence rate.

\begin{lemma}\label{lem.density}
Suppose that $(f,\XX,\nu)$ is a non-uniformly expanding system which admits a Young tower having 
$m\{R>n\}=O\left(\theta^{n^{\beta}}_0\right)$ for some $\theta_0,\beta<1$. Assume also that $\|Df\|_{\infty}<\infty$. Then there exists $\hat\beta<1$, 
such that for any measurable set $A\subset\XX$:
$$\nu(A)=O\left(\exp\{-|\log m(A)|^{\hat\beta}\}\right).$$
\end{lemma}
\begin{proof}
We follow the proof of \cite[Lemma 2.2]{Collet}. In the Young tower construction let $\Lambda_{l}\subset\Lambda$ denote
one of the partition elements, and let $A_{l,j}=A\cap f^{j}(\Lambda_l)$. Let $\tilde{A}_{j,l}\subset\Lambda_l$ be such 
that $f^{j}(A_{j,l})=A\cap f^{j}(\Lambda_l).$ If $K=\|f'\|_{\infty}$, and $R_{l}$ is the return time associated to 
$\Lambda_l$ then we have:
\begin{equation*}
\frac{|\tilde{A}_{j,l}|}{|\Lambda_l|}\leq K^{R_l-j}m(A).
\end{equation*}
Since $\nu_0$ is uniformly equivalent to $m$, we deduce that
$$\nu_0(\tilde{A}_{j,l})\leq K^{R_l-j}m(A)\nu_0(\Lambda_l).$$
Let $g(\cdot)$ be a monotone function with $\lim_{y\to 0}g(y)=\infty$. Then
\begin{equation}
\underset{j,l:K^{R_l-j}<g(m(A))}{\sum}\nu_0(\tilde{A}_{j,l})\leq g(m(A))m(A).
\end{equation}
If $K^{R_l}-j>g(m(A))$, then $R_{l}-j>\log g(m(A))/\log K,$ and we obtain by stretched exponential decay of correlations:
\begin{equation}
\underset{j,l:K^{R_l-j}>g(m(A))}{\sum}\nu_0(\tilde{A}_{j,l})\leq
\underset{R_l>\log g(m(A))/\log K}{\sum} R_lm(\Lambda_l)\leq \exp\{-(\log g(m(A)))^{\beta_1}\},
\end{equation}
where $\beta_1<1$ depends on $\theta_0$ and $K$. If we now choose $g(m(A))=m(A)^{-1/2}$ then the 
result follows.
\end{proof}
The next result gives the quantitative recurrence estimate for $\nu(E_n)$.
\begin{prop}\label{prop.En-sub}
Suppose that $(f,\XX,\nu)$ is a non-uniformly expanding system which admits a Young tower having 
$m\{R>n\}=O\left(\theta^{n^{\beta}}_0\right)$ for some $\beta<1$. Assume also that $\|Df\|_{\infty}<\infty$.
Then for all $\gamma>1$, there exists $\alpha<1$,
such that (H2c) holds in the sense that $\tilde{g}(n)\sim (\log n)^{\gamma}$ implies 
$\nu(E_n)\leq\exp(-(\log n)^{\alpha})$.
\end{prop}
\begin{proof}
The proof extends that of \cite[Proposition 2.3]{Collet} by optimizing the estimates. The proof depends only on the Markov
structure of the tower, the rate of decay of correlations, and the regularity of the invariant density. Let
$F_n(\epsilon)=\{x:\dist(x,f^n(x))<\epsilon\}.$ 
The key estimate derived in the proof of 
\cite[Proposition 2.3]{Collet} is:
\begin{equation}
\nu(F_n(\epsilon))\leq O(1)\left\{\frac{\epsilon}{\delta}+\sum_{s>an}\nu_0(R>n)+n^2\nu(R>bn)+
n^2\nu(R>C\log\delta^{-1})\right\},
\end{equation}
where $a,b$ are fixed numbers in $(0,1)$, and $\delta\in(0,\epsilon)$ can be chosen freely. Using Lemma \ref{lem.density} together with the bound on $\nu_0(R>n)$ we have for all $\delta\in(0,1)$ the inequality
$$\nu(F_n(\epsilon))\leq 0(1)
\left\{\epsilon\delta^{-1}+n^2\exp\{-cn^{\beta_1}\}+n^2\exp\{-|\log \delta|^{\hat\beta}\}\right\},$$
for some $\beta_1<1$ (depending on $\hat\beta$ and $\beta$). Minimizing over $\delta$ we obtain:
\begin{equation}\label{eq.En-sub}
\nu(F_n(\epsilon))\leq O(1)\left\{n^2\exp\{-|\log\epsilon|^{\hat\beta}\}+\exp\{cn^{\beta_1}\}\right\}.
\end{equation}
For small values of $n$ this estimate is of little utility, and we must optimize further, as is done in
\cite[Corollary 2.4]{Collet}. Suppose $\dist(f^j(x),x)\leq\epsilon$, and let $r>1$. Then we have:
$$\dist(f^{rj}(x),x)\leq\sum_{t=0}^{r-1}\dist(f^{tj}(f^j(x)),f^{j}(x))\leq\epsilon\sum_{t=0}^{r-1}K^{tj}\leq \tilde{K}^{rj}\epsilon.$$
Here $K=\|Df\|_{\infty}$, and $\tilde{K}$ is uniformly bounded (independent of $r$ and $j$). Hence
$$F_{j}(\epsilon)\subset F_{rj}(\tilde{K}^{rj}\epsilon).$$ 
Using equation \eqref{eq.En-sub} we obtain:
\begin{equation}\label{eq.En-sub2}
\nu(F_j(\epsilon))\leq O(1)\left((rj)^2\exp\{-|\log(\epsilon \tilde{K}^{rj})|^{\hat\beta}\}+
\exp\{-c(rj)^{\beta_1}\}\right).
\end{equation}
To optimize equation \eqref{eq.En-sub2}, we put $\epsilon=1/n$, and note that $j\leq (\log n)^{\gamma}$ for some
specified $\gamma>1$. Let $a_1>0$ be fixed. We split into two cases: i) $j>a_1\log n$, and ii)
$j\leq a_1\log n$. In case i), we can just apply the bound in equation \eqref{eq.En-sub} and obtain
\begin{equation}
\nu(F_j(1/n))\leq O(1)\left((\log n)^{2}\exp\{-c_1(\log n)^{\hat\beta}\}+\exp\{c_2(\log n)^{\beta_1}\}\right)\leq
O(1)\exp\{-c(\log )^{\beta_2}\},
\end{equation}
for some uniform constants $c,c_1,c_2>0$, $\beta_2<1$. If $j\leq a_1\log n$, then choose 
$r=a_2j^{-1}\log n$ with $a_1,a_2$ fixed so that $a_1<a_2<\log\tilde{K}$. In this case, estimate \eqref{eq.En-sub2}
gives
\begin{equation}
\begin{split}
\nu(F_j(1/n)) &\leq O(1)\left((a_2\log n)^2\exp\{-c_1|\log(n \tilde{K}^{a_2\log n})|^{\hat\beta}\}+
\exp\{-c_2(a_2\log n)^{\beta_1}\}\right)\\
&\leq O(1)\left(\exp\{-c(\log n)^{\beta_3}\}\right),
\end{split}
\end{equation}
for some $\beta_3<1$, and $c>0$. The latter constant $c$ depends on the choice of $a_2$. It follows
that there exists $\alpha<1$ independent of $\gamma$ such that:
$$\nu(E_n)\leq\sum_{j=1}^{\tilde{g}(n)}\nu(F_j(1/n))\leq O(1)\left(\exp\{-(\log n)^{\alpha}\}\right).$$
\end{proof}
\paragraph{Completing the proof of Theorem \ref{thm.stretched} via the blocking argument.} We now estimate
each term that contributes to the error term $\mathcal{E}_n$ as specified in Proposition \ref{prop.blocking}.
We will take $p=q=\sqrt{n}$, and take $t=O((\log n)^{\gamma'})$ for some $1<\gamma'<\gamma$.
Since condition (H2c) holds, we know by Propositions \ref{prop.extreme2} and \ref{prop.En-sub}, that for all $\epsilon>0$
\[
\sum_{j=1}^{(\log n)^{\gamma-\epsilon}} \nu (X_1 > u_n, X_j > u_n)
=
O\left(n^{-1}\exp\{-(\log n)^{\alpha}\}\right) \quad\mbox{for all } n > n_0.
\]
To estimate $\gamma(n,t)$, we use the argument in the proof of Case 2, Proposition \ref{prop.extreme1}.
We obtain
\begin{equation}\label{eq.gamma-sub2}
\gamma(n,t)\leq \exp\{-c|\log\Theta(t)|^{\hat\beta}\},
\end{equation}
where $\hat\beta$ is given in Lemma \ref{lem.density}, and $c>0$ is a uniform constant. 
Using the fact that $\Theta(t)=O(\theta^{t^{\beta}}_0)$, and choosing $t=t_n:=(\log n)^{\gamma'}$
for some $\gamma'>1/(\beta\hat\beta)$ we obtain
\begin{equation}\label{eq.gamma-sub3}
\gamma(n,t_n)\leq \exp\{-\tilde{c}|(\log n)|^{\beta'}\},
\end{equation}
for some $\beta'>1$. Note that for this choice of $t_n$, $\gamma(n,t_n)$ decays to zero at a superpolynomial speed.
Thus this term does not significantly contribute to the error term $\mathcal{E}_n$. 
By Proposition 
\ref{prop.extreme2} we see that the main contribution to the error term comes from $\nu(E_n)$, and we obtain for
some $c>0$,
$$\mathcal{E}_n\leq \exp\{-c|(\log n)|^{\alpha}\},$$  
where $\alpha$ is the constant appearing in Proposition \ref{prop.En-sub}.
Finally, since $\nu$ is absolutely continuous with respect to Lebesgue measure
we have convergence to the Gumbel distribution along the sequence $u_n=u+\log n$.
This completes the proof.



\begin{thebibliography}{10}


\bibitem{Alokley} S. Alokely. Understanding extremes and clustering
in chaotic maps and financial returns data. PhD Thesis, University of Exeter (2015).

\bibitem{AV}
J. Alves. SRB measures for non-hyperbolic systems with multidimensional expansion. 
\emph{Ann. Sci. Ecole Norm. Sup.}, {\bf  33}, (2000), 1–32.

\bibitem{BC2}
M. Benedicks and L. Carleson. The dynamics of the H\'enon map. {\it Annals of Math}., {\bf 133}, (1991), 73-169.

\bibitem{BY1}
M. Benedicks and L.-S. Young.
Markov extensions and decay of correlations for certain H\'enon maps.
{\it Asterisque}.,  No. 261  (2000), xi, 13--56.

\bibitem{CC} J.-R. Chazottes and P. Collet. Poisson approximation for the number of visits to balls in non-uniformly hyperbolic
dynamical systems. {\it Ergodic Theory Dynamical Systems.}  
{\bf 33} (2013), 49-80. 

\bibitem{Coelho} Z. Coelho.  The loss of tightness of time distributions for homeomorphisms of the circle. {\it Trans. AMS.},
{\bf 11}, (2004), 4427-4445.

\bibitem{Coelho-F} Z. Coelho and E. D. Faria. Limit laws of entrance times for homeomporhisms of the circle. {\it Isreal Jour.
Math.}, {\bf 93}, (1006), 93-112.

\bibitem{Collet} P. Collet. Statistics of closest return for some
  non-uniformly  hyperbolic  systems. \emph{Erg. Th. Dyn. Syst.} \textbf{21}
  (2001), 401-420.

\bibitem{diaz-ordaz}
K. D\'iaz-Ordaz. Decay of correlations for non-H\"older observables of expanding Lorenz-like one-dimensional maps.
{\it Discrete Cont. Dyn. Sys.}, {\bf 15}, (2006), 159-176.



\bibitem{Faranda1}
Faranda, D., V. Lucarini, G. Turchetti, and S. Vaienti. 
Numerical convergence of the block-maxima approach to the generalized extreme value distribution. 
J. Stat. Phys.  {\bf 145}  (2011),  no. 5, 1156--1180.

\bibitem{Faranda2}
D. Faranda, V. Lucarini, G. Turchetti, and S. Vaienti. Extreme value distributions for singular measures. {\bf 22}, 023135 (2012),

\bibitem{fefferman}
R. Fefferman. Strong differentiation with respect to measures. \emph{American Journal of Mathematics},
{\bf 103}, (1), (1981), 33--40. 


\bibitem{FF}
A. C. M. Freitas and J. M. Freitas. 
On the link between dependence and independence in extreme value theory for dynamical systems, 
{\it Stat. Probab. Lett.}, {\bf 78}, (2008), 1088-1093.

\bibitem{FFT1}  J. Freitas, A. Freitas and M. Todd. Hitting times and extreme value theory,
\emph{Probab. Theory Related Fields}, {\bf 147(3)}, 675--710, 2010. 

\bibitem{FFT2}  J. Freitas, A. Freitas and M. Todd. Extreme value laws in 
dynamical systems for non-smooth observations.  
{\it J. Stat. Phys.}  {\bf 142}  (2011),  no. 1, 108--126.

\bibitem{FFT3}
A.C.M. Freitas, J.M. Freitas, M. Todd, Extremal index, hitting time statistics and periodicity, 
{\it Adv. Math.}, {\bf 231}, no. 5, 2012, 2626-2665. 

\bibitem{FFT4}
A.C.M. Freitas, J.M. Freitas, and M. Todd. Speed of convergence for laws of rare events and escape rates. {\it Stochastic Processes and their Applications},
{\bf 125}, no. 4 2015, 1653-1687.

\bibitem{FHN}
J. M. Freitas, N. Haydn and M. Nicol. Convergence of rare event point processes to the Poisson for
billiards.

\bibitem{Galambos} J. Galambos. {\em The Asymptotic Theory of Extreme
    Order Statistics}, John Wiley and Sons, 1978.

\bibitem{GP} S. Galatolo and M. J. Pacifico. Lorenz-like flows: exponential decay of correlations for the Poincar\'e map,
logarithm law, quantitative recurrence. Erg. Th. Dyn. Sys. {\bf 30}, no. 6, (2010), 1703-1737.

\bibitem{gouezel}
S. Gouezel. Decay of correlations for nonuniformly expanding systems. \emph{Bull. Soc. math. France},
{\bf 134}, (1), 2006, 1--31.
\bibitem{GW}
J.~Guckenheimer and R.~F. Williams. {Structural stability of Lorenz
  attractors}. \emph{Inst. Hautes \'Etudes Sci. Publ. Math.} \textbf{50}
  (1979) 59--72.

\bibitem{Gupta}
C. Gupta.
Extreme value distributions for some classes of non-uniformly partially hyperbolic dynamical systems.
{\it Ergodic Theory and Dynamical Systems.} {\bf 30}, (3), (2011), 757-771.
\bibitem{GHN} 

C. Gupta, M. P. Holland and M. Nicol. Extreme value theory for dispersing billiards, Lozi maps and 
Lorenz maps. {\it Ergodic Theory and Dynamical Systems,} {\bf 31}, (5), (2011), 1363-1390.  

\bibitem{Hall}
P. Hall. On the rate of convergence of normal extremes. \emph{Journal of Applied Probability}, {\bf 16}, no. 2, (1970),
433-439.


\bibitem{HNPV}
N. Haydn, M. Nicol, T. Persson and S. Vaienti. A note on Borel-Cantelli lemmas for non-uniformly hyperbolic
dynamical systems. Erg. Th. Dyn. Sys.

\bibitem{Haydn-Wasilewska}
N. Haydn and K. Wasilewska. Limiting distribution and error terms for the number of visits of ball in non-uniformly
hyperbolic dynamical systems. Preprint, (2015).


\bibitem{holland-nicol} M. P. Holland and M. Nicol. Speed of convergence to an extreme 
value distribution for non-uniformly hyperbolic dynamical systems. \emph{Stochastics and Dynamics}, {\bf 15}, No. 4 (2015).

\bibitem{HNT}
M. P. Holland and M. Nicol and A. T\"or\"ok.
Extreme value distributions for non-uniformly expanding dynamical systems. {\it Trans. Amer. Math. Soc.}, {\bf 364}, (2012), 661-688.

\bibitem{HVRSB}
M. P. Holland, R. Vitolo, P. Rabassa, A. E. Sterk and H. Broer. Extreme value laws in dynamical systems under physical observables. {\it Physica D: Nonlinear Phenomena}, {\bf 241}, (2012), 497-513.

\bibitem{Kim}
D. H. Kim and K. K. Park. The first return time porperties of an irrational rotation. {\it Proc. AMS.,} {\bf 136,} (11),
(2008), 3941-3951. 

\bibitem{Leadbetter} M.R. Leadbetter, G. Lindgren and H. Rootz\'en. {\em
    Extremes and Related Properties of Random Sequences and Processes},
  Springer-Verlag, (1980).


\bibitem{Licheng}
L. Zhang. Borel Cantelli lemmas and extreme value theory for geometric Lorenz models. Preprint 2015.

\bibitem{Lorenz63}
E.N. Lorenz. {Deterministic non-periodic flow}. \emph{J. Atmosph. Sci.}
  \textbf{20} (1963) 130--141.

\bibitem{mattila}  P.~Mattila. {\em Geometry of sets and measures in Euclidean spaces: fractals and 
rectifiability}. Cambridge Studies in Advanced Mathematics, {\bf 44}, (1995).

\bibitem{Resnick} S. I. Resnick. {\em Extreme values, regular variation,
    and point processes}, Applied Probability Trust, {\bf 4},
  Springer-Verlag, 1987.

\bibitem{Rudin}
W. Rudin, Real and Complex Analysis, 3rd edn. (McGraw Hill, 1987).

\bibitem{Young1} L.-S. Young. Statistical properties of dynamical systems
  with some hyperbolicity. \emph{Ann. of Math.} \textbf{147} (1998)
  585--650.

\bibitem{Young2} L.-S. Young. Recurrence times and rates of mixing.
  \emph{Israel J. Math.} \textbf{110} (1999) 153--188.

\end{thebibliography}
\end{document}